\documentclass[onefignum,onetabnum]{siamart190516}

\usepackage{lipsum}
\usepackage{amsfonts}
\usepackage{graphicx}
\usepackage{epstopdf}
\usepackage{algorithmic}
\ifpdf
  \DeclareGraphicsExtensions{.eps,.pdf,.png,.jpg}
\else
  \DeclareGraphicsExtensions{.eps}
\fi

\newsiamremark{remark}{Remark}
\newsiamremark{hypothesis}{Hypothesis}
\crefname{hypothesis}{Hypothesis}{Hypotheses}
\newsiamthm{claim}{Claim}

\headers{Subgradient algorithm for remote estimation}{C. Kawan, S. Hafstein, and P. Giesl}

\title{A subgradient algorithm for data-rate optimization in the remote state estimation problem\thanks{Submitted to the editors DATE.
\funding{This work was funded by the German Research Foundation (DFG) through the grant ZA 873/4-1.}}}

\author{Christoph Kawan\thanks{Institute of Informatics, LMU Munich, Oettingenstra{\ss}e 67, 80538 M\"{u}nchen, Germany (\email{christoph.kawan@lmu.de}).}
\and Sigurdur Hafstein\thanks{The Science Institute, University of Iceland, Dunhagi 5, 107 Reykjavik, Iceland
  (\email{shafstein@hi.is}).}
\and Peter Giesl\thanks{Department of Mathematics, University of Sussex, Falmer, BN1 9QH, United Kingdom  
  (\email{P.A.Giesl@sussex.ac.uk}).}}

\usepackage{amsopn}

\ifpdf
\hypersetup{
  pdftitle={A subgradient algorithm for data-rate optimization in the remote state estimation problem},
  pdfauthor={C.~Kawan, S.~Hafstein, P.~Giesl}
}
\fi

\newcommand{\N}{\mathbb{N}}%
\newcommand{\R}{\mathbb{R}}%
\newcommand{\Z}{\mathbb{Z}}%
\newcommand{\trn}{^{\scriptscriptstyle \top}}%

\newcommand{\CC}{\mathcal{C}}%
\newcommand{\SC}{\mathcal{S}}%

\newcommand{\tm}{\times}%
\newcommand{\rme}{\mathrm{e}}%
\newcommand{\rmD}{\mathrm{D}}%
\newcommand{\rmO}{\mathrm{O}}%
\newcommand{\rmd}{\mathrm{d}}%
\newcommand{\inner}{\mathrm{int}}%
\newcommand{\res}{\mathrm{res}}%
\newcommand{\ep}{\varepsilon}%
\newcommand{\Gl}{\mathrm{GL}}%
\newcommand{\fa}{\mathfrak{a}}%
\newcommand{\tr}{\mathrm{tr}}%
\newcommand{\Diag}{\mathrm{Diag}}%

\begin{document}

\maketitle

\begin{abstract}
In the remote state estimation problem, an observer tries to reconstruct the state of a dynamical system at a remote location, where no direct sensor measurements are available. The observer only has access to information sent through a digital communication channel with a finite capacity. The recently introduced notion of restoration entropy provides a way to determine the smallest channel capacity above which an observer can be designed that observes the system without a degradation of the initial observation quality. In this paper, we propose a subgradient algorithm to estimate the restoration entropy via the computation of an appropriate Riemannian metric on the state space, which allows to determine the approximate value of the entropy from the time-one map (in the discrete-time case) or the generating vector field (for ODE systems), respectively.
\end{abstract}



\section{Introduction}

A timely and very active field of research within control theory is concerned with the analysis and design of networked control systems. Such systems are often composed of a large number of spatially distributed subsystems which share a common wireless communication network for information transfer. Prominent applications include: cooperative driving of connected cars, the coordinated flight of a swarm of drones, the control of unmanned surveillance and rescue submarines, and  robots  playing football. Designing communication and control protocols for such applications is a major challenge, because they violate some of the standard assumptions in classical control theory due to several imperfections. One of them is a limitation of the available data rate in the employed communication channels. In all practical applications, controllers first compute an estimate of the current state of the system before they determine a control action based on this estimate. These facts motivate the problem of the design of observers which receive sensory data over rate-limited channels. In particular, it is of interest to find the data-rate limit under which such observers can be designed.%

The problem of rate-limited state estimation (or observation) has been studied in \cite{WBr,TM1,TM2} for linear systems, and the minimal required data rate has been characterized as the sum of the logarithms of the unstable eigenvalues (see \cite{MSa,YBa} for comprehensive reviews of related results). For nonlinear deterministic systems, it has been shown in \cite{Sav} that the topological entropy of the system characterizes the rate above which the system can be observed with an arbitrarily small estimation error. Similar studies for nonlinear systems can be found in \cite{KY1,KY2,LMi}, where other variants of the state estimation problem have been studied.%

The analysis and numerical computations in this paper are based on \cite{MPo,MP2,KPM}. The main motivation for these contributions was that the topological entropy has a number of undesirable properties. In particular, the following problems with topological entropy characterize this quantity as a bad choice to base the implementation of control policies on it: (i) its severe non-robustness with respect to uncertain system parameters \cite{Mis} and (ii) the difficulties that come along with its numerical computation (see e.g.~\cite{COH,FJO,Aea}). The main achievement of \cite{MP2} consists in the introduction of a new entropy notion called \emph{restoration entropy}, characterizing the data-rate limit for so-called \emph{regular} or \emph{fine observability}, see \cite[Thm.~9]{MP2}. For the restoration entropy, an explicit formula in terms of the singular values of the linearized system is available \cite[Thm.~11]{MP2}. Moreover, associated observers operating arbitrarily close to the data-rate limit can be designed, which are robust with respect to uncertain parameters in the system.%

In this paper, we develop a numerical algorithm for the computation of restoration entropy and associated Riemannian metrics which can be used for the design of observers. To motivate our numerical algorithm, it is necessary to explain some technical details. Let us assume that the given dynamical system is discrete in time and given by%
\begin{equation*}
  x(t+1) = \phi(x(t)),\quad x(0) \in K,%
\end{equation*}
where $\phi:\R^n \rightarrow \R^n$ is a $C^1$-map and $K \subset \R^n$ a compact forward-invariant set with $K = \overline{\inner\, K}$. The restoration entropy of $\phi$ on $K$ satisfies%
\begin{equation}\label{eq_resent_form}
  h_{\res}(\phi,K) = \lim_{t \rightarrow \infty} \frac{1}{t} \max_{x\in K} \sum_{i=1}^n \max\{0,\log_2 \alpha_i(t,x)\},%
\end{equation}
where $\alpha_1(t,x),\ldots,\alpha_n(t,x)$ are the singular values of the Jacobian matrix $\rmD\phi^t(x)$. The evaluation of the temporal limit in \eqref{eq_resent_form} is closely related to the computation of Lyapunov exponents, which is known to be a difficult problem (see \cite{Sko} for a survey). One approach to the computation of the right-hand side in \eqref{eq_resent_form} is based on a reformulation presented in \cite{KPM} and additionally requires that the matrices $\rmD\phi(x)$, $x\in K$, are invertible. The main observations behind the results of \cite{KPM} are that \eqref{eq_resent_form} still holds if the ordinary singular values $\alpha_i(t,x)$ are replaced by singular values computed with respect to a Riemannian metric on $K$ (see Subsection \ref{subsec_discretetime} for a precise definition), and that the limit in $t$ can be replaced by the infimum over all $t>0$ (due to subadditivity). This leads to%
\begin{equation*}
  h_{\res}(\phi,K) = \inf_{t>0}\frac{1}{t} \max_{x\in K} \sum_{i=1}^n \max\{0,\log_2 \alpha_i^P(t,x)\},%
\end{equation*}
where $P$ is any Riemannian metric on $K$, i.e.~a continuous map from $K$ into the space $\SC^+_n$ of $n\tm n$ positive definite symmetric matrices, and $\alpha_i^P(t,x)$ denote the singular values of $\rmD\phi^t(x)$ computed with respect to $P$. In particular, this implies the estimate%
\begin{equation}\label{eq_resent_ub}
  h_{\res}(\phi,K) \leq \max_{x\in K} \sum_{i=1}^n \max\{0,\log_2 \alpha_i^P(1,x)\}.%
\end{equation}
Using differential-geometric methods introduced by Bochi and Navas in \cite{Boc,BNa}, it has been shown in \cite[Thm.~6]{KPM} that one can choose Riemannian metrics $P$ such that the right-hand side of \eqref{eq_resent_ub} approximates $h_{\res}(\phi,K)$ arbitrarily well:%
\begin{equation}\label{eq_resent_form2}
  h_{\res}(\phi,K) = \inf_P \max_{x\in K} \sum_{i=1}^n \max\{0,\log_2 \alpha_i^P(1,x)\}.%
\end{equation}
Hence, the computation of restoration entropy can be regarded as an infinite-dimensional optimization problem on the space of all Riemannian metrics on $K$. The paper at hand presents a subgradient algorithm designed for solving a constrained version of this optimization problem (and its continuous-time analogue), where we restrict the domain to the class of metrics conformal to a constant metric of the form $P(x) = \rme^{r(x)}p$ with $r(x)$ a polynomial of bounded degree and $p \in \SC^+_n$.%

The fact that this restricted problem can be solved via a subgradient algorithm is due to the observation that the function to be minimized is geodesically convex with respect to $(a,p) \in \R^N \tm \SC^+_n$, where $a$ is the coefficient vector of the polynomial $r(x) = r_a(x)$ and $N = \binom{d+n}{n}$ is the number of coefficients of the polynomial. Hence, we have to deal with a geodesically convex problem on the product space $\R^N \tm \SC^+_n$, which is a complete Riemannian manifold with non-positive sectional curvature, when $\SC^+_n$ is equipped with the standard trace metric \cite{Bha}. The classical subgradient algorithm has been extended to geodesically convex problems on Riemannian manifolds in \cite{FOl,Fea}, and corresponding convergence results have been proven. We have adapted and implemented the algorithm from \cite{FOl,Fea} to estimate the restoration entropy for different systems. The necessary theoretical work in this paper thus consists in a convexity proof for the objective function and the derivation of a formula for its subgradients (both for discrete- and continuous-time systems). The first of these tasks is heavily based on the Riemannian geometry of the space $\SC^+_n$, while the second one also relies on results about generalized derivatives of symmetric singular value or eigenvalue functions \cite{Lew,LSe}.%

We test our algorithm on three examples, two discrete-time systems and one continuous-time system: the H\'enon map with standard parameters, a bouncing ball system, and the Lorenz system with standard parameters. In all case studies, we obtain excellent results which are consistent with the existing theory.%

The paper is organized as follows: Section \ref{sec_prelim} introduces concepts and notation related to the geometry of the space of symmetric positive definite matrices. In Section \ref{sec_convexity}, the central convexity property necessary for the application of the subgradient algorithm is proved. The subsequent Section \ref{sec_reduction} explains how to derive a finite-dimensional optimization problem from the infinite-dimensional one, by restricting the domain to a class of conformal metrics, and how to formulate an associated subgradient algorithm. In Section \ref{sec_subgradient_comp}, explicit formulas for the subgradients in both the discrete- and the continuous-time case are presented. The examples are discussed in Section \ref{sec_examples}, and a list of problems for future research is presented in Section \ref{sec_future}.%

\subsection{Relations to other work}

A similar but simpler problem has been studied in \cite{TGD,Dea}, where the authors consider continuous-time systems given by an ODE%
\begin{equation*}
  \dot{x} = F(x),\quad F \in C^1(\R^n,\R^n),%
\end{equation*}
with associated flow $\phi^t(x)$, and compute the maximal value%
\begin{equation*}
  \bar{\Phi}^* := \max_{x\in K} \limsup_{t \rightarrow \infty} \frac{1}{t}\int_0^t \Phi(\phi^s(x))\, \rmd s%
\end{equation*}
on a compact invariant set $K$ for an observable $\Phi$. Via the variational formulation%
\begin{equation*}
  \bar{\Phi}^* = \inf_{V \in C^1(K)}\max_{x\in K}[\Phi(x) + f(x) \cdot \nabla V(x)],%
\end{equation*}
the maximization problem is transformed into a convex minimization problem on the infinite-dimensional space of $C^1$-functions on $K$, which is then turned into a finite-dimensional problem by restricting the search to SOS (sum of squared) polynomials with a degree bound. SOS programming has also been used to compute contraction metrics to show global stability of an equilibrium for systems with polynomial or rational dynamics \cite{APS}; this is related to our problem as we discuss below and in the conclusions.%

Our work is also related to the computation of extremal Lyapunov exponents. In fact, the number%
\begin{equation*}
  \lim_{t \rightarrow \infty} \frac{1}{t} \max_{x\in K} \sum_{i=1}^n \max\{0,\log \alpha_i(t,x)\}%
\end{equation*}
that we seek to compute, is the maximal Lyapunov exponent of a system induced by the given one on the exterior bundle over $K$. For theoretical results about the approximation of the full Lyapunov spectrum of a linear cocycle via adapted Riemannian metrics, we refer to \cite[Sec.~4]{Boc}. The ideas developed there have been the basis of the proof of formula \eqref{eq_resent_form2}.

Riemannian metrics have also been used to show exponential stability of equilibria and periodic orbits, and to determine subsets of their basins of attraction in the context of contraction metrics, which are tools to show incremental stability \cite{Hah,Leo,Loh,For}. For a Riemannian metric to be such a contraction metric for an equilibrium, the singular values of the linearized system are required to all be negative. For the explicit analytical or numerical computation of these contraction metrics, the restriction to conformal metrics of the form $P(x)=\rme^{r(x)}p$ is often considered.%

In \cite{HK}, an algorithm based on observations made in \cite{PMa} for the computation of an upper bound for the restoration entropy of continuous-time systems was developed. This algorithm used semidefinite optimization to parameterize a Riemannian metric $P:K\to \SC^+_n$ and a Lyapunov-type function $V:K\to\R$, which together deliver an upper bound on the restoration entropy.  Both $P$ and $V$ are continuous and affine on each simplex of triangulations of $K$, and can therefore be parameterized with a finite number of parameters.%

The algorithm achieved its goals in two steps. In the first step, a minimum number $\mu$ satisfying $\lambda_{\max}(x) \leq \mu$ for all $x\in K$ was determined, where $\lambda_{\max}(x)$ denotes the largest generalized eigenvalue of the pair $(A(x),P(x))$ with%
\begin{equation*}
  A(x) := P(x)\rmD F(x) + \rmD F(x)\trn P(x) + \dot P(x).%
\end{equation*}
In this step, a suitable metric $P$ is computed that is used in the second step.%

In the second step, the parameter $Q\in\R$ is minimized under the constraints%
\begin{equation*}
  A(x) - \mu(x)P(x) \preceq 0 \mbox{\quad and \quad} \dot{V}(x) + \widetilde{m}\mu(x) \leq Q\ \ \ \text{for all $x\in K$.}
\end{equation*}
Here, $\widetilde{m}$ is an upper bound on the number of positive generalized eigenvalues of the matrix pairs $(A(x),P(x))$ and the functions $\mu(\cdot)$ and $V(\cdot)$ are continuous and affine on each simplex of a triangulation of $K$. Note that the constant $\mu$ from the first step of the algorithm
serves as an upper bound for the function $\mu(\cdot)$ in this step. One can take $\widetilde{m}=n$, but if a better estimate is available, a lower upper bound $Q/(2\ln(2))$ on the restoration entropy is delivered.%

This algorithm was successfully applied to the Lorenz system, but because of its numerical complexity and lack in maturity of semidefinite solvers, in a simplified form. Indeed, a constant metric $P$ was computed using semidefinite optimization and then the second step could be reduced to a linear programming problem, cf.~\cite[Sec.~3.3]{HK}.%

Finally, further papers about the remote state estimation problem for nonlinear systems and, in particular, restoration entropy include \cite{Kaw,Mea,Vea,Ve2}.%

\section{Preliminaries}\label{sec_prelim}

\subsection{Notation and definitions}

By $\log$ we denote the base-$2$ logarithm. We let $\Z_+ = \{0,1,2,\ldots\}$ be the set of nonnegative integers. The notation $C^k(X,Y)$ is used for the space of all $C^k$-maps from $X$ to $Y$ if $X$ is (a subset of) a smooth manifold and $Y$ is another smooth manifold. We denote by $\SC_n$ the space of all $n\tm n$ real symmetric matrices, and by $\SC^+_n \subset \SC_n$ the subset of positive definite matrices. By $\rmO(n)$, we denote the orthogonal group of $\R^n$ and by $I$ the identity matrix of appropriate dimension.%

Let $M$ be a Riemannian manifold. We write $T_xM$ for the tangent space of $M$ at $x$ and $\langle \cdot,\cdot \rangle_x$ for the inner product on $T_xM$. A subset $C \subset M$ is called \emph{geodesically convex} if for every pair of points $x,y \in C$, there is a unique minimizing geodesic in $M$ joining $x$ and $y$, whose image is contained in $C$. A function $f:M \rightarrow \R$ is called \emph{geodesically convex} if $f \circ \gamma(\theta) \leq (1-\theta)f(\gamma(0)) + \theta f(\gamma(1))$ for every geodesic $\gamma:[0,1] \rightarrow M$ and all $\theta \in [0,1]$.%

\subsection{The space of positive matrices}

We recall some fundamental facts about the geometry of the space $\SC^+_n$. This space, equipped with the \emph{trace metric}%
\begin{equation*}
  \langle v,w \rangle_p := \tr(p^{-1}vp^{-1}w) \mbox{\quad for all\ } p \in \SC^+_n,\ v,w \in T_p\SC^+_n = \SC_n%
\end{equation*}
is a complete Riemannian manifold with non-positive sectional curvature (in fact, a \emph{Hadamard manifold} and also a \emph{symmetric space of non-compact type}), see \cite{Bha,BHa}.%

A first important fact is that the general linear group $\Gl(n,\R)$ acts transitively on $\SC^+_n$ by isometries via%
\begin{equation*}
  g \ast p := gpg\trn,\quad g \in \Gl(n,\R),\ p \in \SC^+_n.%
\end{equation*}

For each pair of points $p,q \in \SC^+_n$, there is a unique minimizing geodesic $\gamma_{pq}(\cdot)$ joining $p$ and $q$, that we always parametrize on $[0,1]$. We use the notation%
\begin{equation*}
  p \#_{\theta}\, q := \gamma_{pq}(\theta),\quad \theta \in [0,1]%
\end{equation*}
and recall that an explicit expression for the geodesic is (see \cite[Thm.~6.1.6]{Bha})%
\begin{equation}\label{eq_geod_form}
  p \#_{\theta}\, q = p^{\frac{1}{2}} [ p^{-\frac{1}{2}} q p^{-\frac{1}{2}} ]^{\theta} p^{\frac{1}{2}}.%
\end{equation}
The formula for the unique geodesic $\gamma_v$ with $\gamma_v(0) = p$ and $\dot{\gamma}_v(0) = v$ is given by%
\begin{equation*}
  \gamma_v(\theta) = p^{\frac{1}{2}} \exp(\theta p^{-\frac{1}{2}}v p^{-\frac{1}{2}}) p^{\frac{1}{2}}%
\end{equation*}
and can be shown by using \eqref{eq_geod_form}, observing that with $q := \gamma_v(1) = p^{\frac{1}{2}} \exp(p^{-\frac{1}{2}}v p^{-\frac{1}{2}}) p^{\frac{1}{2}}$ we obtain%
\begin{align*}
  p \#_{\theta}\, q &= p^{\frac{1}{2}} [ p^{-\frac{1}{2}} p^{\frac{1}{2}} \exp(p^{-\frac{1}{2}}v p^{-\frac{1}{2}}) p^{\frac{1}{2}}p^{-\frac{1}{2}} ]^{\theta} p^{\frac{1}{2}}
	           = p^{\frac{1}{2}} \exp(\theta p^{-\frac{1}{2}}v p^{-\frac{1}{2}}) p^{\frac{1}{2}} = \gamma_v(\theta).%
\end{align*}
Some properties of the geodesics on $\SC^+_n$ that we use are the following \cite{Bha}:%
\begin{itemize}
\item For any positive scalars $a,b$ and matrices $p,q\in\SC^+_n$:%
\begin{equation}\label{eq_s_geodesics}
  (ap) \#_{\theta}\, (bq) = a^{1-\theta} b^{\theta} (p \#_{\theta}\, q).%
\end{equation}
\item For all $g \in \Gl(n,\R)$ and $p,q\in\SC^+_n$:%
\begin{equation}\label{eq_geod_isom}
  g \ast (p \#_{\theta}\, q) = (g \ast p) \#_{\theta}\, (g \ast q).%
\end{equation}
\item For all $p,q \in \SC^+_n$:%
\begin{equation}\label{eq_geod_inv}
  (p \#_{\theta}\, q)^{-1} = p^{-1} \#_{\theta}\, q^{-1}.%
\end{equation}
\end{itemize}
Here, \eqref{eq_geod_isom} and \eqref{eq_geod_inv} directly follow from the facts that $\Gl(n,\R)$ acts on $\SC_n^+$ by isometries and also the matrix inversion, restricted to $\SC_n^+$, is an isometry.%

For any $g \in \Gl(n,\R)$, we let $\alpha_1(g) \geq \ldots \geq \alpha_n(g) > 0$ denote the singular values of $g$, i.e., the eigenvalues of the positive definite symmetric matrix $(gg\trn)^{\frac{1}{2}}$. We define%
\begin{equation*}
  \vec{\sigma}(g) := (\log \alpha_1(g),\ldots,\log \alpha_n(g)).%
\end{equation*}
The function $\vec{\sigma}:\Gl(n,\R) \rightarrow \R^n$ assumes values in the cone $\fa^+ := \{ \xi \in \R^n : \xi_1 \geq \ldots \geq \xi_n \}$ on which we define the partial order%
\begin{equation*}
  \xi \preceq \eta :\Leftrightarrow \left\{\begin{array}{rl}
	                                   \xi_1 + \ldots + \xi_k \leq \eta_1 + \ldots + \eta_k & \mbox{for } k = 1,\ldots,n-1 \\
																	   \xi_1 + \ldots + \xi_k = \eta_1 + \ldots + \eta_k & \mbox{for } k = n,%
																    \end{array}\right.%
\end{equation*}
where $\xi = (\xi_1,\ldots,\xi_n)$ and $\eta = (\eta_1,\ldots,\eta_n)$. From Horn's inequality \cite[Prop.~I.2.3.1]{BLR}, it follows that%
\begin{equation}\label{eq_horn}
  \vec{\sigma}(gh) \preceq \vec{\sigma}(g) + \vec{\sigma}(h) \mbox{\quad for all\ } g,h \in \Gl(n,\R).%
\end{equation}
Finally, we write $\vec{\chi}(g) := (\log |z_1|,\ldots,\log |z_n|)$ for any $g\in\Gl(n,\R)$, where $z_1,\ldots,z_n$ are the eigenvalues of $g$ ordered such that $|z_1| \geq \ldots \geq |z_n|$. By Weyl's inequality \cite[Prop.~I.2.3.3]{BLR}, we have%
\begin{equation}\label{eq_weyl}
  \vec{\chi}(g) \preceq \vec{\sigma}(g) \mbox{\quad for all\ } g \in \Gl(n,\R),%
\end{equation}
where we use that $\log|z_1| + \cdots + \log|z_n| = \log|\det g| = \log \alpha_1(g) + \cdots + \log \alpha_n(g)$.%

\subsection{Restoration entropy and the remote state estimation problem}

\begin{figure}[htbp]
	\begin{center}
		\includegraphics[width=10.00cm,height=2.10cm,angle=0]{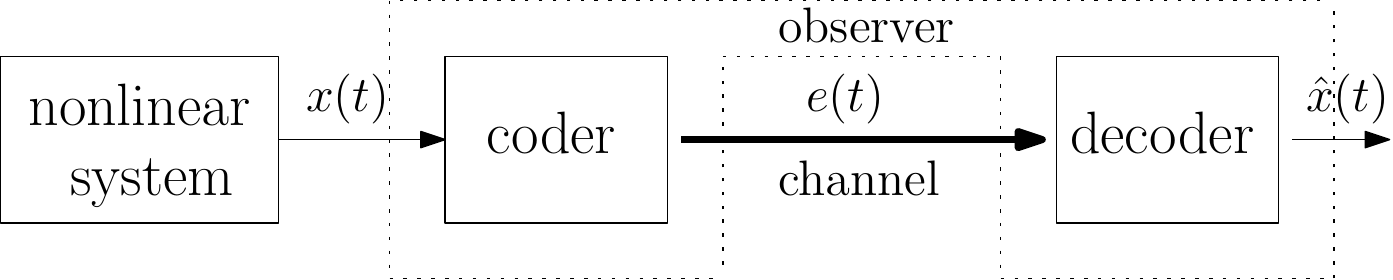}%
	\end{center}
	\caption{Observation at a remote location}\label{fig_obs}
\end{figure}

In this subsection, we briefly explain the remote state estimation problem, which motivates our numerical study. For simplicity, consider a discrete-time system%
\begin{equation}\label{eq_dt_system}
  x(t+1) = \phi(x(t)),\quad x(0) \in K,%
\end{equation}
where $\phi:\R^n \rightarrow \R^n$ is a $C^1$-map and $K \subset \R^n$ a compact forward-invariant set satisfying $K = \overline{\inner\, K}$. Figure \ref{fig_obs} depicts the setup we are interested in. Here, the true state $x(t)$, measured by sensors, is available to a coder and the map $\phi$ as well as an initial estimate $\hat{x}(0)$ are known to both coder and decoder. Based on the knowledge of $x(0),x(1),\ldots,x(t)$ as well as $\phi$ and $\hat{x}(0)$, the coder generates a symbol $e(t)$ from a finite coding alphabet at each time instant $t\in\Z_+$, and sends it over a digital channel to a decoder at a remote location, whose job is to produce an estimate $\hat{x}(t)$ of $x(t)$. The central question is at which data rate (measured in bits per unit of time) these symbols have to be transmitted so that the estimation error $\|x(t) - \hat{x}(t)\|$ can be made small, according to a specified estimation criterion. Here, we concentrate on the concept of \emph{regular observability}, introduced in \cite{MPo}. We say that the system is regularly observed if there are $\delta_*>0$, $G \geq 1$ such that for all $\delta \in (0,\delta_*]$, the implication%
\begin{equation*}
  \|x(0) - \hat{x}(0)\| \leq \delta \quad \Rightarrow \quad \sup_{t\geq0}\|x(t) - \hat{x}(t)\| \leq G\delta%
\end{equation*}
holds for every initial state $x(0) \in K$ and every initial estimate $\hat{x}(0) \in K$, which is known to both coder and decoder at time zero. The smallest information rate above which an observer, satisfying this criterion, can be designed, is given by the restoration entropy $h_{\res}(\phi,K)$ of system \eqref{eq_dt_system}. For the precise definitions of the information rate and  restoration entropy, we refer to \cite{MPo,MP2}. In this paper, we rely on the formula \eqref{eq_resent_form2}, derived in \cite{KPM}, and its continuous-time analogue. In \cite[App.~E]{MPo}, it is explained how a regular observer, associated with a Riemannian metric $P(\cdot)$ on $K$, is designed which operates over a channel whose capacity $c$ satisfies%
\begin{equation*}
  c > \max_{x\in K} \sum_{i=1}^n \max\{0,\log \alpha_i^P(x)\}.%
\end{equation*}
In fact, in \cite[App.~E]{MPo} only constant Riemannian metrics are considered, but the construction works analogously for arbitrary metrics. For practical purposes, it is important that the (analytical or numerical) description of the Riemannian metric allows for an easy approximation of the associated $\delta$-balls, since both coder and decoder have to compute coverings of possibly complicated sets with such balls in real time.%

\section{The central convexity property}\label{sec_convexity}

Let $K \subset \R^n$ be a compact set. If $P,Q \in C^0(K,\SC^+_n)$ are two Riemannian metrics on $K$ and $\theta\in[0,1]$, then $P \#_{\theta}\, Q$ denotes the Riemannian metric defined by%
\begin{equation*}
  (P \#_{\theta}\, Q)(x) := P(x) \#_{\theta}\, Q(x) \mbox{\quad for all\ } x\in K.%
\end{equation*}
Note that this binary operation preserves any regularity requirement on the involved metrics. That is, if both $P$ and $Q$ are of class $C^k$ with $k \in \Z_+ \cup \{\infty\}$, then so is $P \#_{\theta}\, Q$. Indeed, the mapping $\kappa(p,q) := p \#_{\theta}\, q$ from $\SC^+_n \tm \SC^+_n$ to $\SC^+_n$ is of class $C^{\infty}$, and $P \#_{\theta}\, Q = \kappa \circ (P \tm Q)$, where $(P \tm Q)(x) = (P(x),Q(x))$.%

A subset $\CC \subset C^0(K,\SC^+_n)$ will be called \emph{geodesically convex} if for any two $P,Q \in \CC$ it holds that $P \#_{\theta}\, Q \in \CC$ for all $\theta \in [0,1]$. If $\CC$ is geodesically convex, a function $f:\CC \rightarrow \R$ will be called \emph{geodesically convex} if it satisfies%
\begin{equation*}
  f(P \#_{\theta}\, Q) \leq (1 - \theta)f(P) + \theta f(Q)%
\end{equation*}
whenever $P,Q \in \CC$ and $\theta \in [0,1]$. Note that these concepts of geodesic convexity are not the usual ones (which we also use in this paper, namely in Proposition \ref{prop_J_convex} and all that follows), because $\theta \mapsto P \#_{\theta}\, Q$ might not be a geodesic with respect to any (infinite-dimensional) Riemannian metric.%

\subsection{The discrete-time case}\label{subsec_discretetime}

Let $\phi:\R^n \rightarrow \R^n$ be a $C^1$-map. We assume that $K \subset \R^n$ is a compact set satisfying the following properties:%
\begin{itemize}
\item $\phi(K) \subset K$, i.e., $K$ is forward-invariant.%
\item $K$ is the closure of its interior.%
\item The derivative $A(x) := \rmD\phi(x) \in \R^{n\tm n}$ is an invertible matrix for every $x\in K$.%
\end{itemize}
We study the dynamical system%
\begin{equation}\label{eq_dt_sys}
  x(t+1) = \phi(x(t)),\quad x(0) \in K%
\end{equation}
and our goal is to compute the restoration entropy $h_{\res}(\phi,K)$, which by \cite[Thm.~5]{KPM} satisfies%
\begin{equation*}
  h_{\res}(\phi,K) = \inf_{P \in C^0(K,\SC^+_n)} \max_{x\in K}\sum_{i=1}^n \max\{0,\log\alpha_i^P(x)\}.%
\end{equation*}

To this end, for any Riemannian metric $P \in C^0(K,\SC^+_n)$, we define%
\begin{equation*}
  \Sigma^P(x) := \sum_{i=1}^n \max\{0,\log \alpha_i^P(x)\} \mbox{\quad for all\ } x \in K,%
\end{equation*}
where $\alpha_1^P(x) \geq \ldots \geq \alpha_n^P(x) > 0$ are the eigenvalues of $[B(x)\trn B(x)]^{\frac{1}{2}}$ with%
\begin{equation*}
  B(x) := P(\phi(x))^{\frac{1}{2}}A(x)P(x)^{-\frac{1}{2}}.%
\end{equation*}
That is, $\alpha_i^P(x)$ are the ordinary singular values of $B(x)$ or the singular values of $A(x)$, regarded as a linear operator between the inner product spaces $(\R^n,\langle P(x) \cdot,\cdot\rangle)$ and $(\R^n,\langle P(\phi(x))\cdot,\cdot\rangle)$, see \cite[Lem.~5]{KPM}. Moreover, we put%
\begin{equation*}
  \Sigma(P) := \max_{x\in K}\Sigma^P(x) \mbox{\quad for all\ } P \in C^0(K,\SC^+_n).%
\end{equation*}

\begin{lemma}
For each $x\in K$, the functional $P \mapsto \Sigma^P(x)$ from $C^0(K,\SC^+_n)$ to $\R_+$ is continuous with respect to the uniform topology on its domain.
\end{lemma}

\begin{proof}
Fix $P \in C^0(K,\SC^+_n)$ and let $\ep>0$. We have to show that there exists $\delta>0$ such that%
\begin{equation*}
  \sup_{z\in K}\|P(z) - Q(z)\| \leq \delta \quad \Rightarrow \quad |\Sigma^P(x) - \Sigma^Q(x)| \leq \ep.%
\end{equation*}
First note that by continuous dependence of the singular values on the matrix, we can choose $\alpha>0$ small enough such that%
\begin{align}\label{eq_alphaep}
  \|P(\phi(x))^{\frac{1}{2}}A(x)P(x)^{-\frac{1}{2}} - Q(\phi(x))^{\frac{1}{2}}A(x)Q(x)^{-\frac{1}{2}}\| \leq \alpha \Rightarrow |\Sigma^P(x) - \Sigma^Q(x)| \leq \ep.%
\end{align}
Choosing $\delta$ small enough leads to $\max\{\|P(x) - Q(x)\|,\|P(\phi(x)) - Q(\phi(x))\|\}$ being as small as desired. Since the matrix $\hat{q}^{\frac{1}{2}}A(x)q^{-\frac{1}{2}}$ depends continuously on $(q,\hat{q}) \in \SC^+_n \tm \SC^+_n$, therefore the assumption in \eqref{eq_alphaep} can be satisfied.
\end{proof}

The following lemma is the key result of this paper, which enables us to formulate our optimization problem as a convex problem.%

\begin{lemma}\label{lem_convexity_dt}
For all $P,Q \in C^0(K,\SC^+_n)$, every $x \in K$ and $\theta \in [0,1]$, it holds that%
\begin{align}\label{eq_convexity}
\begin{split}
 &\vec{\sigma}( [P(\phi(x)) \#_{\theta}\, Q(\phi(x))]^{\frac{1}{2}}A(x)[P(x) \#_{\theta}\, Q(x)]^{-\frac{1}{2}} ) \\
	&\preceq (1 - \theta)\vec{\sigma}( P(\phi(x))^{\frac{1}{2}}A(x)P(x)^{-\frac{1}{2}} ) + \theta\vec{\sigma}(Q(\phi(x))^{\frac{1}{2}}A(x)Q(x)^{-\frac{1}{2}} ).%
\end{split}
\end{align}
As a consequence, $P \mapsto \Sigma^P(x)$ is geodesically convex, i.e.%
\begin{equation*}
  \Sigma^{P \#_{\theta}\, Q}(x) \leq (1 - \theta)\Sigma^P(x) + \theta\Sigma^Q(x).%
\end{equation*}
\end{lemma}

\begin{proof}
The identity%
\begin{equation*}
  \Sigma^P(x) = \max_{0 \leq k \leq n}\sum_{i=1}^k \log\alpha_i^P(x),%
\end{equation*}
where $\sum_{i=1}^0 \ldots := 0$, implies that it suffices to prove the relation \eqref{eq_convexity} for geodesic convexity of $P \mapsto \Sigma^P(x)$. By a standard convexity argument (see Remark \ref{rem_conv_arg} for details), it suffices to prove this for $\theta=\frac{1}{2}$. To this end, let $g(x) := P(x)^{\frac{1}{2}}$, which yields $g(x)^{-1} \ast P(x) \equiv I$. Then, putting $Q'(x) := g(x)^{-1} \ast Q(x)$, we obtain%
\begin{align*}
  &\vec{\sigma}\bigl( [P(\phi(x)) \#_{\frac{1}{2}}\, Q(\phi(x))]^{\frac{1}{2}}A(x)[P(x) \#_{\frac{1}{2}}\, Q(x)]^{-\frac{1}{2}} \bigr) \\
	&= \vec{\sigma}\bigl( [(g(\phi(x)) \ast I) \#_{\frac{1}{2}}\, Q(\phi(x))]^{\frac{1}{2}} A(x) [ (g(x) \ast I) \#_{\frac{1}{2}}\, Q(x)]^{-\frac{1}{2}} \bigr) \\
	&\stackrel{\eqref{eq_geod_isom}}{=} \vec{\sigma}\bigl( [g(\phi(x)) \ast (I \#_{\frac{1}{2}}\, Q'(\phi(x)))]^{\frac{1}{2}} A(x) [g(x) \ast (I \#_{\frac{1}{2}}\, Q'(x))]^{-\frac{1}{2}} \bigr) \\
	&\stackrel{\eqref{eq_geod_form}}{=} \vec{\sigma}\bigl( [g(\phi(x)) \ast Q'(\phi(x))^{\frac{1}{2}}]^{\frac{1}{2}} A(x) [g(x) \ast Q'(x)^{\frac{1}{2}}]^{-\frac{1}{2}} \bigr) \\
	&= \frac{1}{2} \vec{\chi}\bigl( [g(x) \ast Q'(x)^{\frac{1}{2}}]^{-\frac{1}{2}} A(x)\trn [g(\phi(x)) \ast Q'(\phi(x))^{\frac{1}{2}}] A(x) [g(x) \ast Q'(x)^{\frac{1}{2}}]^{-\frac{1}{2}} \bigr) \\
	&= \frac{1}{2} \vec{\chi}\bigl( [g(x) \ast Q'(x)^{\frac{1}{2}}]^{-1} A(x)\trn [g(\phi(x)) \ast Q'(\phi(x))^{\frac{1}{2}}] A(x) \bigr) \\
  &= \frac{1}{2} \vec{\chi}\bigl( P(x)^{-\frac{1}{2}} Q'(x)^{-\frac{1}{2}} P(x)^{-\frac{1}{2}} A(x)\trn P(\phi(x))^{\frac{1}{2}} Q'(\phi(x))^{\frac{1}{2}} P(\phi(x))^{\frac{1}{2}} A(x) \bigr) \\
	&= \frac{1}{2} \vec{\chi}\bigl( Q'(x)^{-\frac{1}{2}} P(x)^{-\frac{1}{2}} A(x)\trn P(\phi(x))^{\frac{1}{2}} Q'(\phi(x))^{\frac{1}{2}} P(\phi(x))^{\frac{1}{2}} A(x) P(x)^{-\frac{1}{2}} \bigr) \\
	&= \frac{1}{2} \vec{\chi}\bigl( P(x)^{-\frac{1}{2}} A(x)\trn P(\phi(x))^{\frac{1}{2}} Q'(\phi(x))^{\frac{1}{2}} P(\phi(x))^{\frac{1}{2}} A(x) P(x)^{-\frac{1}{2}} Q'(x)^{-\frac{1}{2}} \bigr) \\
	&\stackrel{\eqref{eq_weyl}}{\preceq} \frac{1}{2}\vec{\sigma}\bigl( [P(x)^{-\frac{1}{2}} A(x)\trn P(\phi(x))^{\frac{1}{2}}] \cdot [Q'(\phi(x))^{\frac{1}{2}} P(\phi(x))^{\frac{1}{2}} A(x) P(x)^{-\frac{1}{2}} Q'(x)^{-\frac{1}{2}}] \bigr) \\
	&\stackrel{\eqref{eq_horn}}{\preceq} \frac{1}{2}\vec{\sigma}\bigl( P(\phi(x))^{\frac{1}{2}} A(x) P(x)^{-\frac{1}{2}} \bigr) + \frac{1}{2}\vec{\sigma}\bigl(Q'(\phi(x))^{\frac{1}{2}} P(\phi(x))^{\frac{1}{2}} A(x) P(x)^{-\frac{1}{2}} Q'(x)^{-\frac{1}{2}}\bigr).%
\end{align*}
The last summand (without the factor $\frac{1}{2}$ in front) can be written as%
\begin{align*}
  & \vec{\sigma}\bigl( [P(\phi(x))^{-\frac{1}{2}} Q(\phi(x)) P(\phi(x))^{-\frac{1}{2}}]^{\frac{1}{2}}P(\phi(x))^{\frac{1}{2}}A(x) P(x)^{-\frac{1}{2}}[P(x)^{-\frac{1}{2}}Q(x)P(x)^{-\frac{1}{2}}]^{-\frac{1}{2}} \bigr) \\
	&= \vec{\sigma}\bigl( P(\phi(x))^{-\frac{1}{2}}[P(\phi(x)) \#_{\frac{1}{2}}\, Q(\phi(x))] A(x) [P(x)^{-1} \#_{\frac{1}{2}}\, Q(x)^{-1}] P(x)^{\frac{1}{2}} \bigr).%
\end{align*}
We now introduce the following abbreviations:%
\begin{equation*}
  p := P(x), \quad q := Q(x), \quad \hat{p} := P(\phi(x)), \quad \hat{q} := Q(\phi(x)), \quad a := A(x).%
\end{equation*}
Then, using \eqref{eq_geod_inv}, we can write the last term above as%
\begin{align*}
  &\vec{\sigma}\bigl( \hat{p}^{-\frac{1}{2}} [\hat{p} \#_{\frac{1}{2}}\, \hat{q}]  a [p \#_{\frac{1}{2}}\, q]^{-1} p^{\frac{1}{2}} \bigr) \\
	&= \frac{1}{2}\vec{\chi}\bigl(p^{\frac{1}{2}} [p \#_{\frac{1}{2}}\, q]^{-1} a\trn [\hat{p} \#_{\frac{1}{2}}\, \hat{q}] \hat{p}^{-\frac{1}{2}} \hat{p}^{-\frac{1}{2}} [\hat{p} \#_{\frac{1}{2}}\, \hat{q}]  a [p \#_{\frac{1}{2}}\, q]^{-1} p^{\frac{1}{2}}\bigr).%
\end{align*}
Observe that%
\begin{align*}
  [\hat{p} \#_{\frac{1}{2}}\, \hat{q}] \hat{p}^{-\frac{1}{2}} \hat{p}^{-\frac{1}{2}} [\hat{p} \#_{\frac{1}{2}}\, \hat{q}] = \hat{p}^{\frac{1}{2}} [ \hat{p}^{-\frac{1}{2}} \hat{q} \hat{p}^{-\frac{1}{2}} ]^{\frac{1}{2}} [\hat{p}^{-\frac{1}{2}} \hat{q}\hat{p}^{-\frac{1}{2}}]^{\frac{1}{2}} \hat{p}^{\frac{1}{2}} = \hat{p}^{\frac{1}{2}} \hat{p}^{-\frac{1}{2}} \hat{q} \hat{p}^{-\frac{1}{2}} \hat{p}^{\frac{1}{2}} = \hat{q}.%
\end{align*}
Hence,%
\begin{align*}
&\vec{\sigma}\bigl( \hat{p}^{-\frac{1}{2}} [\hat{p} \#_{\frac{1}{2}}\, \hat{q}]  a [p \#_{\frac{1}{2}}\, q]^{-1} p^{\frac{1}{2}} \bigr) = \frac{1}{2}\vec{\chi}\bigl(p^{\frac{1}{2}} [p \#_{\frac{1}{2}}\, q]^{-1} a\trn \hat{q}  a [p \#_{\frac{1}{2}}\, q]^{-1} p^{\frac{1}{2}} \bigr) \\
	&= \frac{1}{2}\vec{\chi}\bigl( [ p^{\frac{1}{2}} q^{-1} p^{\frac{1}{2}} ]^{\frac{1}{2}} p^{-\frac{1}{2}} a\trn \hat{q} a p^{-\frac{1}{2}} [ p^{\frac{1}{2}} q^{-1} p^{\frac{1}{2}} ]^{\frac{1}{2}} \bigr) \\
	&= \frac{1}{2}\vec{\chi}\bigl( p^{\frac{1}{2}} q^{-1} p^{\frac{1}{2}} p^{-\frac{1}{2}}a\trn \hat{q} a p^{-\frac{1}{2}} \bigr) \\
	&= \frac{1}{2}\vec{\chi}\bigl( q^{-1} a\trn \hat{q} a \bigr) \\
	&= \frac{1}{2}\vec{\chi}\bigl( q^{-\frac{1}{2}} q^{-\frac{1}{2}} a\trn \hat{q}^{\frac{1}{2}} \hat{q}^{\frac{1}{2}} a \bigr) \\
	&= \frac{1}{2}\vec{\chi}\bigl( [q^{-\frac{1}{2}} a\trn \hat{q}^{\frac{1}{2}} ] \cdot [\hat{q}^{\frac{1}{2}} a q^{-\frac{1}{2}} ] \bigr) \\
	&= \frac{1}{2}\vec{\chi}\bigl( [\hat{q}^{\frac{1}{2}} a q^{-\frac{1}{2}}]\trn \cdot [\hat{q}^{\frac{1}{2}} a q^{-\frac{1}{2}} ] \bigr) \\
	&= \vec{\sigma}\bigl( \hat{q}^{\frac{1}{2}} a q^{-\frac{1}{2}} \bigr) = \vec{\sigma}\bigl( Q(\phi(x))^{\frac{1}{2}} A(x) Q(x)^{-\frac{1}{2}} \bigr),
\end{align*}
which yields the desired inequality.
\end{proof}

\begin{remark}\label{rem_conv_arg}
Recall the argument showing that it suffices to check the convexity condition for $\theta = \frac{1}{2}$: If a functional $J$ on $C^0(K,\SC^+_n)$ satisfies $J(P \#_{\frac{1}{2}}\, Q) \leq \frac{1}{2}J(P) + \frac{1}{2}J(Q)$ for all $P,Q$, then%
\begin{align*}
  J(P \#_{\frac{1}{4}}\, Q) &= J(P \#_{\frac{1}{2}}\, [P \#_{\frac{1}{2}}\, Q])
	\leq \frac{1}{2}J(P) + \frac{1}{2}J(P \#_{\frac{1}{2}}\, Q) \\
	&\leq \frac{1}{2}J(P) + \frac{1}{2}\left(\frac{1}{2}J(P) + \frac{1}{2}J(Q)\right) = \frac{3}{4}J(P) + \frac{1}{4}J(Q).%
\end{align*}
In a similar fashion, we can can verify the convexity condition for every number of the form $\theta = \frac{k}{2^n}$ with $n$ a positive integer and $0 \leq k \leq 2^n$. For all other values of $\theta \in [0,1]$ it follows by continuity, since the set $\{\frac{k}{2^n}\}$ is dense in $[0,1]$. Here we use that $P \mapsto J(P) := \Sigma^P(x)$ (for fixed $x$) is continuous with respect to the uniform topology on $C^0(K,\SC^+_n)$ and also $\theta \mapsto P \#_{\theta}\, Q$ is continuous.%
\end{remark}

\begin{corollary}
The functional $\Sigma:C^0(K,\SC^+_n) \rightarrow \R_+$ satisfies%
\begin{equation*}
  \Sigma(P \#_{\theta}\, Q) \leq (1-\theta)\Sigma(P) + \theta\Sigma(Q)%
\end{equation*}
for all $P,Q \in C^0(K,\SC^+_n)$ and $\theta \in [0,1]$. That is, $\Sigma$ is geodesically convex.
\end{corollary}

\subsection{The continuous-time case}

Consider an ODE%
\begin{equation}\label{eq_ct_sys}
  \dot{x} = F(x)%
\end{equation}
with a $C^1$-vector field $F:\R^n \rightarrow \R^n$. Let $\phi^t(x)$ denote the induced flow and assume that $K \subset \R^n$ is a compact forward-invariant set which is the closure of its interior. We write%
\begin{equation*}
  A(x) := \rmD F(x) \mbox{\quad for all\ } x \in K.%
\end{equation*}

For any Riemannian metric $P \in C^1(K,\SC^+_n)$,\footnote{In fact, we only need that the orbital derivatives $\dot{P}(x)$ exist, hence less regularity would be sufficient.} we put%
\begin{equation*}
  \hat{\Sigma}^P(x) := \sum_{i=1}^n \max\{0,\zeta_i^P(x)\} \mbox{\quad for all\ } x \in K,%
\end{equation*}
where $\zeta_1^P(x) \geq \ldots \geq \zeta_n^P(x)$ are the solutions of the algebraic equation%
\begin{equation}\label{eq_ct_alg_eq}
  \det[ P(x)A(x) + A(x)\trn P(x) + \dot{P}(x) - \lambda P(x) ] = 0.%
\end{equation}
Moreover, we put%
\begin{equation*}
  \hat{\Sigma}(P) := \max_{x\in K}\hat{\Sigma}^P(x) \mbox{\quad for all\ } P \in C^1(K,\SC^+_n).%
\end{equation*}
According to \cite[Thm.~11]{KPM}, the restoration entropy of system \eqref{eq_ct_sys} on $K$ satisfies%
\begin{equation*}
  h_{\res}(f,K) = \frac{1}{2\ln(2)} \inf_{P \in C^1(K,\SC^+_n)} \max_{x\in K}\sum_{i=1}^n \max\{0,\zeta_i^P(x)\}.%
\end{equation*}
The numbers $\zeta_1^P(x),\ldots,\zeta_n^P(x)$ can be regarded as infinitesimal counterparts to the singular values employed in the discrete-time case.%

The following lemma is the continuous-time counterpart to Lemma \ref{lem_convexity_dt}.%

\begin{lemma}\label{lem_convexity_ct}
Let $P,Q \in C^1(K,\SC^+_n)$. Then for every $x\in K$ and every $\theta \in [0,1]$ the following inequality holds:%
\begin{equation*}
  \hat{\Sigma}^{P \#_{\theta}\, Q}(x) \leq (1 - \theta)\hat{\Sigma}^P(x) + \theta\hat{\Sigma}^Q(x).%
\end{equation*}
That is, $\hat{\Sigma}$ is geodesically convex.
\end{lemma}

\begin{proof}
Consider for each $t \geq 0$ the time-$t$ map $\phi^t$, which is well-defined on $K$ (due to compactness and forward-invariance). Let us put%
\begin{equation*}
  \Sigma^P(x,k,t) := \sum_{i=1}^k \log \alpha_i^P(x;\phi^t)%
\end{equation*}
for all $x\in K$, $k \in \{1,\ldots,n\}$ and $t \geq 0$, where $\alpha_1^P(x;\phi^t) \geq \ldots \geq \alpha_n^P(x;\phi^t)$ are the eigenvalues of $[B_t(x)\trn B_t(x)]^{\frac{1}{2}}$ with%
\begin{equation*}
  B_t(x) := P(\phi^t(x))^{\frac{1}{2}}\rmD\phi^t(x)P(x)^{-\frac{1}{2}}.%
\end{equation*}
From the proof of Lemma \ref{lem_convexity_dt}, we know that%
\begin{equation}\label{eq_dt_convexity}
  \Sigma^{P \#_{\theta} Q}(x,k,t) \leq (1 - \theta)\Sigma^P(x,k,t) + \theta\Sigma^Q(x,k,t)%
\end{equation}
for all $\theta \in [0,1]$ and $P,Q \in C^0(K,\SC^+_n)$. We claim that%
\begin{equation}\label{eq_from_dt_to_ct}
  \sum_{i=1}^k \zeta^P_i(x) = \ln(2) \lim_{t \rightarrow 0^+}\frac{1}{t}\Sigma^P(x,k,t) \mbox{\quad for all\ } x \in K.%
\end{equation}
Since $\hat{\Sigma}^P(x) = \max_{0 \leq k \leq n}\sum_{i=1}^k \zeta^P_i(x)$, this together with \eqref{eq_dt_convexity} implies the assertion of the lemma. To prove the claim, first observe that $\Sigma^P(x,k,0) = 0$ implies%
\begin{equation*}
  \lim_{t \rightarrow 0^+}\frac{1}{t}\Sigma^P(x,k,t) = \frac{\rmd}{\rmd t}\Bigl|_{t=0^+} \Sigma^P(x,k,t),%
\end{equation*}
provided that the derivative exists. To compute the derivative (and show its existence), we use \cite[Cor.~23]{KPM} which tells us that%
\begin{equation*}
  \frac{\rmd}{\rmd t}\Bigl|_{t=0^+} \log \alpha_i^P(x;\phi^t) = \frac{1}{\ln(2)} \lambda_i\Bigl(\frac{\rmd}{\rmd t}\Bigl|_{t=0^+}[B_t(x) + B_t(x)\trn]\Bigr),%
\end{equation*}
where $\lambda_1(p) \geq \ldots \geq \lambda_n(p)$ denote the eigenvalues of a real symmetric matrix $p$. We have%
\begin{equation*}
  \frac{\rmd}{\rmd t}\Bigl|_{t=0^+} B_t(x) = \Bigl[\frac{\rmd}{\rmd t}\Bigl|_{t=0^+} P(\phi^t(x))^{\frac{1}{2}}\Bigr] P(x)^{-\frac{1}{2}} + P(x)^{\frac{1}{2}} \rmD F(x) P(x)^{-\frac{1}{2}}%
\end{equation*}
and%
\begin{equation*}
  \frac{\rmd}{\rmd t}\Bigl|_{t=0^+} B_t(x)\trn = P(x)^{-\frac{1}{2}} \rmD F(x)\trn P(x)^{\frac{1}{2}} + P(x)^{-\frac{1}{2}} \Bigl[\frac{\rmd}{\rmd t}\Bigl|_{t=0^+} P(\phi^t(x))^{\frac{1}{2}}\Bigr].%
\end{equation*}
Write ${\bf r}(p) := p^{\frac{1}{2}}$. Then it is easy to see that%
\begin{equation*}
  \rmD {\bf r}(P(x)) \dot{P}(x) P(x)^{\frac{1}{2}} + P(x)^{\frac{1}{2}}\rmD{\bf r}(P(x)) \dot{P}(x) = \dot{P}(x)%
\end{equation*}
which, after multiplication with $P(x)^{-\frac{1}{2}}$ from both sides, yields%
\begin{align*}
  &\frac{\rmd}{\rmd t}\Bigl|_{t=0^+}[B_t(x) + B_t(x)\trn] \\
	&= P(x)^{\frac{1}{2}} \rmD F(x) P(x)^{-\frac{1}{2}} + P(x)^{-\frac{1}{2}} \rmD F(x)\trn P(x)^{\frac{1}{2}} + P(x)^{-\frac{1}{2}}\dot{P}(x) P(x)^{-\frac{1}{2}} \\
	&= P(x)^{-\frac{1}{2}}\bigl[ P(x) \rmD F(x) + \rmD F(x)\trn P(x) + \dot{P}(x) \bigr] P(x)^{-\frac{1}{2}}.
\end{align*}
The claim then follows from the observation that the solutions of \eqref{eq_ct_alg_eq} are precisely the eigenvalues of the matrix above.
\end{proof}

\begin{corollary}
The functional $\hat{\Sigma}:C^1(K,\SC^+_n) \rightarrow \R_+$ satisfies%
\begin{equation*}
  \hat{\Sigma}(P \#_{\theta}\, Q) \leq (1-\theta)\hat{\Sigma}(P) + \theta\hat{\Sigma}(Q)%
\end{equation*}
for all $P,Q \in C^1(K,\SC^+_n)$ and $\theta \in [0,1]$. That is, $\hat{\Sigma}$ is geodesically convex.
\end{corollary}

\section{Reduction to a finite-dimensional problem}\label{sec_reduction}

In this section, we explain how we can formulate a finite-dimensional geodesically convex optimization problem over a space of metrics on $K$ conformal to a constant metric. Moreover, we introduce the Riemannian subgradient algorithm that can solve such a problem.%

\subsection{Formulation of the optimization problem}

We want to develop numerical algorithms to solve the optimization problems%
\begin{equation}\label{eq_min_problems}
  \min_{P \in C^0(K,\SC^+_n)} \Sigma(P) \mbox{\quad and \quad} \min_{P \in C^1(K,\SC^+_n)} \hat{\Sigma}(P).%
\end{equation}
To transform these infinite-dimensional problems into feasible finite-dimensional convex optimization problems, we have to restrict the domain to a geodesically convex subset which can be described by finitely many real parameters. At the same time, we must be aware that for the remote state estimation problem we have to be able to find a minimal ball covering of the set $K$ with geodesic balls in the (optimal) metric $P$. Hence, we should search for an optimal metric within a class of metrics that has a ``nice'' analytic description. The following class of conformal metrics seems to be a good candidate:%
\begin{align*}
  \mathrm{C}_d(K) &:= \{ P \in C^{\infty}(K,\SC^+_n) : P(x) \equiv \rme^{r(x)}p \mbox{ for some } p \in \SC^+_n, \\
	                          &\qquad\qquad\qquad\qquad\qquad\qquad r(x) \mbox{ a polynomial of degree} \leq d \}.%
\end{align*}
Another reason for the choice of this class of metrics is that in several examples, where an analytical expression for an optimal metric is known, it is of the form $\rme^{r(x)}p$, see for instance \cite[Sec.~4]{KPM} and \cite[Sec.~8]{MP2}.%

\begin{lemma}
The set $\mathrm{C}_d(K)$ is geodesically convex, since%
\begin{equation}\label{eq_conf_expr}
  (\rme^{r(x)}p) \#_{\theta}\, (\rme^{s(x)}q) = \rme^{(1-\theta)r(x) + \theta s(x)} p \#_{\theta}\, q \mbox{\quad for all\ } \theta \in [0,1].%
\end{equation}
\end{lemma}

\begin{proof}
This follows immediately from the identity \eqref{eq_s_geodesics}.
\end{proof}

Since a polynomial of degree $d$ in $n$ variables has $\binom{d+n}{n} = \frac{(d+n)!}{d!n!}$ coefficients, the class $\mathrm{C}_d(K)$ can be parameterized by $\binom{d+n}{n} + \frac{n(n+1)}{2}$ real parameters, where the parameter space is%
\begin{equation*}
  P(d,n) := \R^{\binom{d+n}{n}} \tm \SC^+_n.%
\end{equation*}
We equip the Euclidean factor of $P(d,n)$ with the Euclidean metric and $P(d,n)$ with the associated product metric. Then, as is well-known, the unique geodesic between two points $(a,p)$ and $(b,q)$ in $P(d,n)$, parameterized on $[0,1]$, has the form%
\begin{equation}\label{eq_product_geodesic}
  \gamma_{(a,p)(b,q)}(\theta) = ((1-\theta)a + \theta b,p \#_{\theta}\, q).%
\end{equation}
Given a parameter vector $a \in \R^{\binom{d+n}{n}}$, we write $r_a(x)$ for the associated polynomial (where we assume that the assignment $a \mapsto r_a(x)$ is defined in a consistent way).%

\begin{proposition}\label{prop_J_convex}
Given the discrete-time system \eqref{eq_dt_sys}, the function%
\begin{equation*}
  J:P(d,n) \rightarrow \R_+,\quad J(a,p) := \Sigma(\rme^{r_a(\cdot)}p)%
\end{equation*}
is geodesically convex.
\end{proposition}

\begin{proof}
We have to show that%
\begin{equation*}
  J(\gamma_{(a,p)(b,q)}(\theta)) \leq (1 - \theta)J(a,p) + \theta J(b,q)%
\end{equation*}
for any $a,b\in \R^{\binom{d+n}{n}}$, $p,q \in \SC^+_n$ and $\theta \in [0,1]$:%
\begin{align*}
  J(\gamma_{(a,p)(b,q)}(\theta)) &\stackrel{\eqref{eq_product_geodesic}}{=} J((1-\theta)a + \theta b,p \#_{\theta}\, q) \\
	&= \Sigma( \rme^{(1-\theta)r_a(\cdot) + \theta r_b(\cdot)} p \#_{\theta}\, q) \\
	&\stackrel{\eqref{eq_conf_expr}}{=} \Sigma( [\rme^{r_a(\cdot)}p] \#_{\theta}\, [\rme^{r_b(\cdot)}q] ) \\
	&\leq (1 - \theta)\Sigma(\rme^{r_a(\cdot)}p) + \theta\Sigma(\rme^{r_b(\cdot)}q) \\
	&= (1 - \theta) J(a,p) + \theta J(b,q),
\end{align*}
where the inequality follows from Lemma \ref{lem_convexity_dt}.
\end{proof}

In the continuous-time case, we can analogously introduce a geodesically convex functional $\hat{J}$ on $P(d,n)$ via $\hat{\Sigma}$.%

The algorithms developed in the next sections solve the minimization problems%
\begin{equation*}
  \min_{(a,p) \in P(d,n)} J(a,p) \mbox{\quad and \quad} \min_{(a,p) \in P(d,n)} \hat{J}(a,p),%
\end{equation*}
respectively, instead of \eqref{eq_min_problems}.%

\subsection{Solution via the subgradient algorithm}

The paper \cite{FOl} introduces a subgradient algorithm to solve geodesically convex optimization problems on Riemannian manifolds. Before we go into details of this algorithm, recall the following facts for a geodesically convex function $f:M \rightarrow \R$ defined on a complete Riemannian manifold $M$ \cite{Udr}:%
\begin{itemize}
\item $f$ is locally Lipschitz continuous \cite[Cor.~3.10]{Udr}.%
\item Given $x\in M$, a vector $s \in T_xM$ is called a \emph{subgradient} of $f$ at $x$ if for any geodesic $\gamma$ of $M$ with $\gamma(0) = x$ the following inequality holds:%
\begin{equation*}
  (f \circ \gamma)(\theta) \geq f(x) + \theta \langle s,\dot{\gamma}(0) \rangle_x \mbox{\quad for all\ } \theta \geq 0.%
\end{equation*}
The set of all subgradients, denoted by $\partial f(x)$, is called the \emph{subdifferential} of $f$ at $x$. The subdifferential at any point $x$ is nonempty, convex and compact \cite[Thm.~4.5 and 4.6]{Udr}.%
\end{itemize}

The \emph{subgradient algorithm} consists of the following steps: Given a sequence $(\theta_k)_{k\in\N}$ of step sizes with $\theta_k > 0$ for all $k$:%
\begin{enumerate}
\item[(0)] Initialize. Choose $p_1 \in M$ and compute some $s_1 \in \partial f(p_1)$. Put $k := 1$.%
\item[(1)] If $s_k = 0$, stop. Otherwise, compute the geodesic $\gamma_{v_k}$ with $\gamma_{v_k}(0) = p_k$, $\dot{\gamma}_{v_k}(0) = v_k$, $v_k = -s_k/|s_k|$.%
\item[(2)] Put $p_{k+1} := \gamma_{v_k}(\theta_k)$.%
\item[(3)] Compute some $s_{k+1} \in \partial f(p_{k+1})$. Put $k := k+1$ and go to (1).%
\end{enumerate}

For the convergence of the sequence $p_k$ to a minimizer, a proper choice of the step sizes $\theta_k$ is necessary, and it is an important assumption that the sectional curvatures of $M$ are uniformly bounded from below. In our case, this is guaranteed by Lemma \ref{lem_product_curv} in the appendix, which shows that the parameter space $P(d,n)$ for the conformal metrics in $C_d(K)$ satisfies this property.%

The \emph{diminishing or exogeneous step size rule} requires to choose the step sizes $\theta_k$ such that%
\begin{equation*}
  \sum_{k=1}^{\infty}\theta_k = \infty \mbox{\quad and \quad} \sum_{k=1}^{\infty}\theta_k^2 < \infty.%
\end{equation*}
Assuming that the sectional curvatures of $M$ are uniformly bounded below, with such a choice (typically, $\theta_k = a/(k+b)$ with $a>0$, $b \geq 0$), \cite[Thm.~3.2]{Fea} guarantees that%
\begin{equation*}
  \liminf_{k\rightarrow\infty} f(p_k) = \inf f%
\end{equation*}
and that $p_k$ converges to a minimizer if a minimizer exists. Provided that more information about the function $f$ is available, other step size rules can be used, which come with estimates for the speed of convergence, see \cite[Thm.~3.3]{Fea}.%

\section{Computation of subgradients}\label{sec_subgradient_comp}

In this section, we provide a method to compute subgradients of our objective functions $J$ and $\hat{J}$, respectively.%

\subsection{The discrete-time case}

For the computation of subgradients, we have to recall further facts about geodesically convex functions $f:M \rightarrow \R$. First, we define unilateral directional derivatives.%

\begin{definition}
Let $\gamma:[0,1] \rightarrow M$ be a geodesic with $\gamma(0) = x$ and $\dot{\gamma}(0) = v$. Then%
\begin{equation}\label{eq_directional_der}
  \partial_v f(x) := \lim_{\theta \rightarrow 0^+}\frac{f(\gamma(\theta)) - f(x)}{\theta}%
\end{equation}
is called the \emph{unilateral directional derivative} of $f$ at $x$ in direction $v$.
\end{definition}

Then we have the following facts:%
\begin{enumerate}
\item[(F1)] \cite[Thm.~4.2]{Udr}: The unilateral directional derivative \eqref{eq_directional_der} exists and satisfies%
\begin{equation*}
  \partial_v f(x) = \inf_{\theta>0}\frac{f(\gamma(\theta)) - f(x)}{\theta}.%
\end{equation*}
Moreover, $v \mapsto \partial_v f(x)$ is convex and positively homogeneous with $\partial_0 f(x) = 0$, $-\partial_{-v}f(x) \leq \partial_vf(x)$.%
\item[(F2)] \cite[Thm.~4.8]{Udr}: A vector $s \in T_xM$ is a subgradient of $f$ at $x$ if and only if%
\begin{equation*}
  \partial_vf(x) \geq \langle s,v \rangle_x \mbox{\quad for all\ } v \in T_xM.%
\end{equation*}
\item[(F3)] (Trivial) If $f = \sup_{\alpha\in A} f_{\alpha}$ for a family of geodesically convex functions $f_{\alpha}$, and $f(x) = f_{\alpha}(x)$ for some $x\in M$ and $\alpha \in A$, then $\partial f_{\alpha}(x) \subset \partial f(x)$.%
\end{enumerate}

We will use these facts to compute a subgradient of the geodesically convex function $J:P(d,n) \rightarrow \R_+$, defined in Proposition \ref{prop_J_convex}. We can write $J$ as%
\begin{equation*}
  J(a,p) = \max_{x\in K}\max_{0 \leq k \leq n}\sum_{i=1}^k \log \alpha_i\bigl(\rme^{\frac{1}{2}[r_a(\phi(x)) - r_a(x)]}p^{\frac{1}{2}}A(x)p^{-\frac{1}{2}}\bigr).%
\end{equation*}
From the proofs of Lemma \ref{lem_convexity_dt} and Proposition \ref{prop_J_convex}, we can see that the inner functions%
\begin{equation*}
  J_{k,x}(a,p) := \sum_{i=1}^k \log \alpha_i\bigl(\rme^{\frac{1}{2}[r_a(\phi(x)) - r_a(x)]}p^{\frac{1}{2}}A(x)p^{-\frac{1}{2}}\bigr)%
\end{equation*}
are geodesically convex for each $x\in K$ and $k = 1,\ldots,n$. Hence, by (F3), the task of computing a subgradient for $J$ at $(a,p)$ splits into the following three subtasks:%
\begin{enumerate}
\item[(T1)] Solve the maximization problem%
\begin{equation}\label{eq_dt_maximization}
  \max_{x\in K}\Bigl[\max_{0 \leq k \leq n}\sum_{i=1}^k \log \alpha_i\bigl(\rme^{\frac{1}{2}[r_a(\phi(x)) - r_a(x)]}p^{\frac{1}{2}}A(x)p^{-\frac{1}{2}}\bigr)\Bigr]%
\end{equation}
leading to a (not necessarily unique) maximizer $x^* \in K$.%
\item[(T2)] Solve the maximization problem%
\begin{equation*}
   \max_{0 \leq k \leq n}\sum_{i=1}^k \log \alpha_i\bigl(\rme^{\frac{1}{2}[r_a(\phi(x^*)) - r_a(x^*)]}p^{\frac{1}{2}}A(x^*)p^{-\frac{1}{2}}\bigr)%
\end{equation*}
which is trivial, since only finitely many quantities are involved. This leads to a maximizer $k^*$ (where $k^* = 0$ is allowed as a trivial case).%
\item[(T3)] Compute a subgradient of $J_{k^*,x^*}$ at $(a,p)$.%
\end{enumerate}
For (T1), there is no general method, since this optimization problem may not have nice properties (such as convexity). However, it is lower-dimensional than the original minimization problem we want to solve and the existence of a maximizer is guaranteed.%

For task (T3), we first note that $J_{k^*,x^*}(a,p)$ can be written as%
\begin{equation*}
  J_{k^*,x^*}(a,p) = \underbrace{\frac{k^*}{2\ln(2)}[r_a(\phi(x^*)) - r_a(x^*)]}_{=: J^1_{k^*,x^*}(a)} + \underbrace{\sum_{i=1}^{k^*} \log \alpha_i(p^{\frac{1}{2}}A(x^*)p^{-\frac{1}{2}})}_{=: J^2_{k^*,x^*}(p)}.%
\end{equation*}
By definition, a subgradient of $J_{k^*,x^*}$ at $(a,p)$ is a tangent vector $s \in T_{(a,p)}P(d,n) \cong \R^{\binom{d+n}{n}} \tm \SC_n$, $s = (s_1,s_2)$, such that%
\begin{equation*}
  J_{k^*,x^*}(\gamma_1(\theta),\gamma_2(\theta)) \geq J_{k^*,x^*}(a,p) + \theta \langle (s_1,s_2), (\dot{\gamma}_1(0),\dot{\gamma}_2(0)) \rangle_{(a,p)}%
\end{equation*}
for all $\theta \geq 0$ and for every geodesic $\gamma(\theta) = (\gamma_1(\theta),\gamma_2(\theta))$ in $P(d,n)$ with $\gamma(0) = (a,p)$. This is equivalent to%
\begin{equation*}
  J^1_{k^*,x^*}(\gamma_1(\theta)) + J^2_{k^*,x^*}(\gamma_2(\theta)) \geq J^1_{k^*,x^*}(a) + J^2_{k^*,x^*}(p) + \theta \langle s_1,\dot{\gamma}_1(0) \rangle_a + \theta \langle s_2,\dot{\gamma}_2(0) \rangle_p.%
\end{equation*}
Hence, we can split the task of computing a subgradient of $J_{k^*,x^*}$ at $(a,p)$ into the following subtasks:%
\begin{enumerate}
\item[(T3.1)] Compute a subgradient $s_1 \in \R^{\binom{d+n}{n}}$ of $J^1_{k^*,x^*}$ at $a$.%
\item[(T3.2)] Compute a subgradient $s_2 \in \SC_n$ of $J^2_{k^*,x^*}$ at $p$.%
\end{enumerate}
Task (T3.1) is trivial, since $a \mapsto J^1_{k^*,x^*}(a)$ is a linear function. To see how its constant gradient $\nabla J^1_{k^*,x^*}(a)$ looks like, we consider an example. Let $n = d = 2$ and write%
\begin{equation*}
  r_a(x) = a_0 + a_1 x_1 + a_2 x_2 + a_{12}x_1x_2 + a_{11} x_1^2 + a_{22}x_2^2.%
\end{equation*}
Then%
\begin{align*}
  J^1_{k^*,x^*}(a) &= \frac{k^*}{2\ln(2)}\bigl[0,\phi(x^*)_1 - x^*_1,\phi(x^*)_2 - x^*_2,\phi(x^*)_1\phi(x^*)_2 - x^*_1x^*_2,\\
	&\qquad \phi(x^*)_1^2 - (x^*_1)^2,\phi(x^*)_2^2 - (x^*_2)^2 \bigr] \cdot \left[\begin{array}{c} a_0 \\ a_1 \\ a_2 \\ a_{12} \\ a_{11} \\ a_{22} \end{array}\right] = \nabla J^1_{k^*,x^*}(a) \cdot \left[\begin{array}{c} a_0 \\ a_1 \\ a_2 \\ a_{12} \\ a_{11} \\ a_{22} \end{array}\right].%
\end{align*}
Since the gradient, if it exists, is the only subgradient, we are done with (T3.1).%

To compute a subgradient of $J^2_{k^*,x^*}$ at $p$, we decompose this function as follows:%
\begin{equation*}
  J^2_{k^*,x^*} = f \circ \alpha \circ \zeta,%
\end{equation*}
where%
\begin{align}\label{eq_fct_def}
\begin{split}
  \zeta:\SC^+_n \rightarrow \Gl(n,\R),\quad p &\mapsto p^{\frac{1}{2}}A(x^*)p^{-\frac{1}{2}}, \\
	\alpha:\Gl(n,\R) \rightarrow \R^n,\quad g &\mapsto (\alpha_1(g),\ldots,\alpha_n(g)), \\
	f:\R^n \rightarrow \R \cup \{-\infty\},\quad x &\mapsto \sum_{i=1}^{k^*} \log \hat{x}_i,
\end{split}
\end{align}
where $\hat{x} = (\hat{x}_1,\ldots,\hat{x}_n)$ is the vector that is derived from $x$ by first taking the absolute value of each component and then putting these nonnegative numbers in non-increasing order. The following lemma yields some crucial properties of the functions $\zeta$ and $f$.%

\begin{lemma}\label{lem_fct_props}
The following holds:%
\begin{enumerate}
\item[(i)] The function $\zeta$ is differentiable and its derivative satisfies%
\begin{equation*}
  \rmD\zeta(p)h = XA(x^*)p^{-\frac{1}{2}} - p^{\frac{1}{2}}A(x^*)p^{-\frac{1}{2}}Xp^{-\frac{1}{2}}%
\end{equation*}
for all $h \in T_p\SC^+_n = \SC_n$, where $X$ is the unique solution of the Lyapunov equation%
\begin{equation*}
  p^{\frac{1}{2}}X + Xp^{\frac{1}{2}} = h.%
\end{equation*}
\item[(ii)] The function $f$ is absolutely symmetric, i.e., for every signed $n \tm n$ permutation matrix\footnote{A signed permutation matrix is a matrix $P$ such that $|P|$ (componentwise defined absolute value) is a permutation matrix.} $P$ it holds that $f(Px) = f(x)$ for all $x \in \R^n$.%
\item[(iii)] Let $x \in \R^n$ with $x_1 \geq x_2 \geq \ldots \geq x_{k^*} > x_{k^*+1} \geq x_{k^*+2} \geq \ldots \geq x_n > 0$. Then $f$ is differentiable at $x$ with%
\begin{equation*}
  \nabla f(x) = \sum_{i=1}^{k^*} \frac{e_i\trn}{\ln(2)x_i} = \frac{1}{\ln(2)}\Bigl(\frac{1}{x_1},\ldots,\frac{1}{x_{k^*}},0,\ldots,0\Bigr),%
\end{equation*}
where $e_i$ is the $i$-th unit vector in $\R^n$.
\end{enumerate}
\end{lemma}

\begin{proof}
(i) The differentiability of $\zeta$ follows from the differentiability of the functions ${\bf r}(p) := p^{\frac{1}{2}}$ and ${\bf i}(p) := p^{-1}$ (well-known). The derivatives of these functions satisfy%
\begin{equation*}
  \rmD {\bf i}(p)h = -p^{-1}hp^{-1},\quad p^{\frac{1}{2}}\rmD {\bf r}(p)h + \rmD {\bf r}(p)h p^{\frac{1}{2}} = h.%
\end{equation*}
Now let $\gamma:\R \rightarrow \SC^+_n$ be a $C^1$-curve with $\gamma(0) = p$ and $\dot{\gamma}(0) = h$ for some $h \in \SC_n$. Then%
\begin{align*}
  \rmD\zeta(p)h &= \frac{\rmd}{\rmd \theta}\Bigl|_{\theta=0} \zeta(\gamma(\theta)) = \frac{\rmd}{\rmd \theta}\Bigl|_{\theta=0}{\bf r}(\gamma(\theta))A(x^*) {\bf i} \circ {\bf r}(\gamma(\theta)) \\
	&= \Bigl[\frac{\rmd}{\rmd \theta}\Bigl|_{\theta=0}{\bf r}(\gamma(\theta))\Bigr] A(x^*) p^{-\frac{1}{2}} + p^{\frac{1}{2}}A(x^*)\Bigl[\frac{\rmd}{\rmd \theta}\Bigl|_{\theta=0}{\bf i} \circ {\bf r}(\gamma(\theta))\Bigr] \\
	&= [\rmD r(p)h] A(x^*) p^{-\frac{1}{2}} - p^{\frac{1}{2}} A(x^*) p^{-\frac{1}{2}} [\rmD r(p)h] p^{-\frac{1}{2}}.%
\end{align*}

(ii) Obvious.%

(iii) If $z\in\R^n$ is a vector of sufficiently small norm, then all of the numbers $x_1 + z_1,\ldots,x_n + z_n$ are positive and the numbers $x_1 + z_1,\ldots,x_{k^*} + z_{k^*}$ are strictly larger than each of the numbers $x_{k^*+1} + z_{k^*+1},\ldots,x_n + z_n$. Hence,%
\begin{equation*}
  f(x + z) = \sum_{i=1}^{k^*} \log (x_i + z_i).%
\end{equation*}
This immediately yields the claimed formula for the gradient $\nabla f(x)$.
\end{proof}

We will use a result from \cite[Thm.~7.1]{LSe}. For its formulation, we need the following definition.%

\begin{definition}
Given a Euclidean space $E$, a function $f:E \rightarrow [-\infty,+\infty]$ and a point $x\in E$ at which $f$ is finite, an element $y\in E$ is called a \emph{regular subgradient} of $f$ at $x$ if it satisfies%
\begin{equation*}
  f(x + z) \geq f(x) + \langle y,z \rangle_E + o(z)%
\end{equation*}
with $\lim_{z \rightarrow 0}\|z\|^{-1}o(z) = 0$. An element $y$ of $E$ is called a \emph{limiting subgradient} of $f$ at $x$ if there is a sequence of points $x_n \in E$ converging to $x$ such that $f(x_n) \rightarrow f(x)$ and a sequence of regular subgradients $y_n$ at $x_n$ such that $y_n \rightarrow y$. The set of all limiting subgradients is called the \emph{limiting subdifferential} and is denoted by $\bar{\partial}f(x)$.
\end{definition}

\begin{theorem}\label{thm_lewis_sendov}
Let $f:\R^n \rightarrow [-\infty,+\infty]$ be an absolutely symmetric function. Then the limiting subdifferential of $f \circ \alpha$, with $\alpha$ as defined in \eqref{eq_fct_def}, at a matrix $X$ is given by%
\begin{equation*}
  \bar{\partial} (f \circ \alpha)(X) = \{ U\trn\Diag\, \bar{\partial} f(\alpha(X))V : (U,V) \in \rmO(n,n)^X \},%
\end{equation*}
where%
\begin{equation*}
  \rmO(n,n)^X := \{ (U,V) \in \rmO(n) \tm \rmO(n) : U\trn\Diag(\alpha(X))V = X \}%
\end{equation*}
and $\Diag(x)$ is the diagonal matrix with entries $x_1,\ldots,x_n$ on the diagonal for any $x = (x_1,\ldots,x_n)\in\R^n$.%
\end{theorem}

We will use the above theorem in the following way to compute a subgradient of $J^2_{k^*,x^*} = f \circ \alpha \circ \zeta$ at $p \in \SC^+_n$:%
\begin{itemize}
\item Assume that there exists a regular subgradient $S$ of $f \circ \alpha$ at $X := \zeta(p)$ (this assumption will be justified below). Then we know that%
\begin{equation}\label{eq_reg_sub}
  f(\alpha(X + Z)) \geq f(\alpha(X)) + \tr[ S\trn Z ] + o(Z),\quad \lim_{Z \rightarrow 0}\frac{o(Z)}{\|Z\|} = 0,%
\end{equation}
since $\langle S,Z \rangle = \tr[S\trn Z]$ is the Euclidean inner product in $\R^{n\tm n}$.%
\item Observe that $-f$ is a convex function on the open and convex set $\{ x \in \R^n : x > 0,\ x_{k^*} > x_{k^* + 1} \}$ (its Hessian is positive semidefinite). Then, according to \cite[Prop.~6.2]{LSe}, $f \circ \alpha$ is differentiable at $X$ if and only if $f$ is differentiable at $\alpha(X)$ and its gradient is given by%
\begin{equation*}
   \nabla (f \circ \alpha)(X) = U\trn \Diag(\nabla f(\alpha(X))) V,%
\end{equation*}
where $(U,V) \in \rmO(n) \tm \rmO(n)$ such that $X = U\trn \Diag(\alpha(X))V$. From Lemma \ref{lem_fct_props}(iii), we conclude that the inequality%
\begin{equation}\label{eq_singval_gap}
  \alpha_{k^*}(X) > \alpha_{k^*+1}(X),%
\end{equation}
implies that $f \circ \alpha$ is differentiable at $X$ with%
\begin{equation}\label{eq_euclidean_subgradient}
  \nabla (f \circ \alpha)(X) = \frac{1}{\ln(2)} {U}\trn \Diag\Bigl(\frac{1}{\alpha_1(X)},\ldots,\frac{1}{\alpha_{k^*}(X)},0,\ldots,0\Bigr){V},%
\end{equation}
and hence, in this case the only (regular or limiting) subgradient of $f \circ \alpha$ is the gradient: $S = \nabla(f \circ \alpha)(X)$. Since \eqref{eq_singval_gap} is generically satisfied, we will use formula \eqref{eq_euclidean_subgradient} in the rest of the paper.%
\item We fix a geodesic $\gamma:[0,1] \rightarrow \SC^+_n$ with $\gamma(0) = p$ and define%
\begin{equation*}
  Z(\theta) := \zeta(\gamma(\theta)) - \zeta(p) \in \R^{n \tm n}.%
\end{equation*}
Then, by \eqref{eq_reg_sub}%
\begin{align*}
  J^2_{k^*,x^*}(\gamma(\theta)) &= (f \circ \alpha \circ \zeta)(\gamma(\theta)) = f(\alpha(X + Z(\theta))) \\
	&\geq J^2_{k^*,x^*}(p) + \tr[ S\trn Z(\theta) ] + o(Z(\theta)).%
\end{align*}
Now we write%
\begin{equation*}
  Z(\theta) = \rmD\zeta(p)\dot{\gamma}(0)\theta + o(\theta),\quad \lim_{\theta \rightarrow 0}\frac{o(\theta)}{\theta} = 0,%
\end{equation*}
leading to%
\begin{align*}
  J^2_{k^*,x^*}(\gamma(\theta)) - J^2_{k^*,x^*}(p) \geq \theta \cdot \tr[ S\trn \rmD\zeta(p)\dot{\gamma}(0)] + \tr [S\trn o(\theta)] + o(Z(\theta)).%
\end{align*}
Dividing both sides by $\theta$ and letting $\theta \rightarrow 0^+$ yields%
\begin{equation}\label{eq_sigmakx_ineq}
  \partial_{\dot{\gamma}(0)}J^2_{k^*,x^*}(p) \geq \tr[ S\trn \rmD\zeta(p)\dot{\gamma}(0)],%
\end{equation}
where we use that%
\begin{equation*}
  \lim_{\theta \rightarrow 0}\frac{o(Z(\theta))}{\theta} = \lim_{\theta \rightarrow 0}\frac{o(Z(\theta))}{\|Z(\theta)\|}\frac{\|Z(\theta)\|}{\theta} = 0%
\end{equation*}
from the formula for $Z(\theta)$. Since $\gamma$ was arbitrary, \eqref{eq_sigmakx_ineq} implies%
\begin{equation*}
  \partial_v J^2_{k^*,x^*}(p) \geq \tr[ S\trn \rmD\zeta(p)v ] \mbox{\quad for all\ } v \in \SC_n.%
\end{equation*}
Hence, if we can find some $s_2 \in \SC_n$ such that%
\begin{equation*}
  \tr[S\trn \rmD\zeta(p)v] = \langle s_2,v \rangle_p = \tr[p^{-1}s_2p^{-1}v] \mbox{\quad for all\ } v \in \SC_n,%
\end{equation*}
then (F2) implies that $s_2$ is a subgradient of $J^2_{k^*,x^*}$ at $p$.
\item The existence of $s_2$ is guaranteed by the Riesz representation theorem. A way to construct $s_2$ is to choose an orthonormal basis $\{e_i\}$ of the inner product space $(\SC_n, \langle \cdot,\cdot \rangle_p)$ and put%
\begin{equation}\label{eq_subgradient_finalform}
  s_2 := \sum_i \tr[S\trn \rmD\zeta(p)e_i]e_i.%
\end{equation}
An orthonormal basis can be obtained by applying the Gram-Schmidt process to the standard basis consisting of symmetric matrices with only one nonzero entry, equal to $1$, above the diagonal.%
\end{itemize}

We have solved task (T3.2){, since $s_2$ is the desired subgradient.}%

\subsection{The continuous-time case}

In the continuous-time case, we have to minimize the function%
\begin{equation*}
  \hat{J}:P(d,n) \rightarrow \R_+,\quad \hat{J}(a,p) = \hat{\Sigma}(\rme^{r_a(\cdot)}p).%
\end{equation*}
Recall that $\hat{\Sigma}(P)$, $P \in C^1(K,\SC^+_n)$, was defined via the eigenvalues of the symmetric matrices%
\begin{equation*}
  P(x)^{-\frac{1}{2}}\bigl[P(x)A(x) + A(x)\trn P(x) + \dot{P}(x)\bigr]P(x)^{-\frac{1}{2}},\quad x \in K,%
\end{equation*}
where $A(x) = \rmD F(x)$, see \eqref{eq_ct_sys}. For a metric of the form $P(x) = \rme^{r(x)}p$, this becomes%
\begin{equation*}
  p^{\frac{1}{2}} A(x) p^{-\frac{1}{2}} + p^{-\frac{1}{2}}A(x)\trn p^{\frac{1}{2}} + \dot{r}(x) I.%
\end{equation*}
We can thus compute a subgradient along the following steps:%
\begin{enumerate}
\item[(T1)] Solve the maximization problem%
\begin{equation*}
  \max_{x\in K} \sum_{i=1}^n \max\{0,\zeta_i^{\rme^{r_a(\cdot)}p}(x)\} = \max_{x\in K} \max_{0 \leq k \leq n} \sum_{i=1}^k \zeta_i^{\rme^{r_a(\cdot)}p}(x),%
\end{equation*}
leading to a maximizer $x^* \in K$.%
\item[(T2)] Solve the maximization problem%
\begin{equation*}
  \max_{0 \leq k \leq n} \sum_{i=1}^k \zeta_i^{\rme^{r_a(\cdot)}p}(x^*),%
\end{equation*}
leading to a maximizer $k^* \in \{0,1,\ldots,n\}$.%
\item[(T3)] Compute a subgradient $(s_1,s_2)$ of%
\begin{equation*}
  \hat{J}(a,p) = \underbrace{k^*\dot{r}_a(x^*)}_{=:\hat{J}^1_{k^*,x^*}(a)} + \underbrace{\sum_{i=1}^{k^*} \lambda_i\bigl(p^{\frac{1}{2}} A(x^*) p^{-\frac{1}{2}} + p^{-\frac{1}{2}}A(x^*)\trn p^{\frac{1}{2}}\bigr)}_{=:\hat{J}^2_{k^*,x^*}(p)},%
\end{equation*}
where $\lambda_1(h) \geq \ldots \geq \lambda_n(h)$ denote the eigenvalues of a symmetric matrix $h$.%
\end{enumerate}

The computation of a subgradient for $\hat{J}^1_{k^*,x^*}(a)$ is again simple. For instance, consider the case $n = d = 2$. Then%
\begin{equation*}
  r_a(x) = a_0 + a_1 x_1 + a_2 x_2 + a_{12}x_1x_2 + a_{11} x_1^2 + a_{22}x_2^2.%
\end{equation*}
We thus obtain%
\begin{align*}
  \dot{r}_a(x) &= a_1 \dot{x}_1 + a_2 \dot{x}_2 + a_{12}(\dot{x}_1x_2 + x_1\dot{x}_2) + 2 a_{11} x_1 \dot{x}_1 + 2 a_{22} x_2 \dot{x}_2 \\
	             &= a_1 F_1(x) + a_2 F_2(x) + a_{12}(F_1(x) x_2 + x_1 F_2(x)) + 2 a_{11} x_1 F_1(x) + 2 a_{22} x_2 F_2(x).%
\end{align*}
Hence, the gradient exists and is given by%
\begin{equation*}
  \nabla \hat{J}^1_{k^*,x^*}(a) = {k^* \cdot [ 0,F_1(x^*), F_2(x^*), F_1(x^*) x^*_2 + x^*_1 F_2(x^*), 2x^*_1F_1(x^*), 2x^*_2F_2(x^*) ]}.%
\end{equation*}

For the computation of a subgradient of $\hat{J}^2_{k^*,x^*}(p)$, we write this function as%
\begin{equation*}
  \hat{J}^2_{k^*,x^*} = g \circ \lambda \circ \hat{\zeta},%
\end{equation*}
where%
\begin{align*}
  \hat{\zeta}:\SC^+_n \rightarrow \SC_n,\quad p &\mapsto p^{\frac{1}{2}} A(x^*) p^{-\frac{1}{2}} + p^{-\frac{1}{2}}A(x^*)\trn p^{\frac{1}{2}}, \\
	\lambda:\SC_n \rightarrow \R^n,\quad X &\mapsto (\lambda_1(X),\ldots, \lambda_n(X)), \\
	g:\R^n \rightarrow \R_+,\quad x  &\mapsto \sum_{i=1}^{k^*} \hat{x}_i,%
\end{align*}
where $\hat{x}$ is defined as before. Relying on \cite[Thm.~7.2]{Lew} and \cite[Lem.~6.3]{LSe}, we can compute a Euclidean subgradient of $g \circ \lambda$ at $X := \hat{\zeta}(p)$ as%
\begin{equation*}
   S = U \Diag(\underbrace{1,\ldots,1}_{k^*},\underbrace{0,\ldots,0}_{n-k^*}) U\trn,%
\end{equation*}
where $U$ is an orthogonal matrix such that $X = U \Diag(\lambda_1(X),\ldots,\lambda_n(X)) U\trn$, and assuming that $\lambda_{k^*}(X) > \lambda_{k^* + 1}(X)$. With the same reasoning as in the discrete-time case, we find that a subgradient $s_2$ of $\hat{J}^2_{k^*,x^*}$ at $p$ must satisfy%
\begin{equation*}
  \tr[S\trn \rmD\hat{\zeta}(p)v] = \langle s_2,v \rangle_p = \tr[p^{-1}s_2p^{-1}v] \mbox{\quad for all\ } v \in \SC_n.%
\end{equation*}
The derivative of $\hat{\zeta}$ is given by%
\begin{align*}
  \rmD\hat{\zeta}(p)v = Y A(x^*) p^{-\frac{1}{2}} - p^{\frac{1}{2}}A(x^*)p^{-\frac{1}{2}}Yp^{-\frac{1}{2}} - p^{-\frac{1}{2}}Yp^{-\frac{1}{2}}A(x^*)\trn p^{\frac{1}{2}} + p^{-\frac{1}{2}}A(x^*)\trn Y,%
\end{align*}
where $Y$ is the solution of the Lyapunov equation $p^{\frac{1}{2}}Y + Yp^{\frac{1}{2}} = v$. A subgradient is then given by%
\begin{equation*}
  s_2 = \sum_i \tr[S\trn \rmD\hat{\zeta}(p)e_i]e_i,%
\end{equation*}
where $\{e_i\}$ is an {orthonormal basis} of the inner product space $(\SC_n,\langle\cdot,\cdot,\rangle_p)$.%

\section{Examples}\label{sec_examples}

The most demanding numerical task in all the computations is the maximization (T1). In general, one does not have much information about the function to be maximized that helps finding the maximum. Therefore, we used brute-force search on a cube $\prod_{i=1}^n [a_i,b_i] = K$ in appropriate coordinates. To this end, we generate a regular grid $G\subset K$ with points of the form $(z_1,z_2,\ldots,z_n)$, where%
\begin{equation*}
  z_i = a_i + jh_i,\ h_i=\frac{b_i-a_i}{N_i-1}\ \ \ \text{and}\ \ \ j=0,\ldots,N_i-1 \text{ for all } i=1\ldots,n.%
\end{equation*}
Here, the $N_i$s are numbers that determine the density of the grid. After an $x^*_{\rm pre}\in G$ that maximizes the expression in (T1) on $G$ has been determined, we refine the search for a maximizer $x^*$ on a finer grid on the cube $\left(x^*_{\rm pre} + \prod_{i=1}^n [-h_i/2,h_i/2]\right) \cap K$ around $x^*_{\rm pre}$. In principle, this grid can be chosen independently of $G$, however, we used a scaled down version of $G$. In particular, the refined grid has the same number of points as $G$ if $x^*_{\rm pre} + \prod_{i=1}^n [-h_i/2,h_i/2]\subset K$.%

The algorithms were programmed in C++ using the Armadillo library and run on AMD ThreadRipper 3990X (64 cores@2.9GHz). The code is described and published in \cite{KHG2021SOFTX}. Version v1.1 of the code at https://github.com/shafstein/EntEstSG was used to produce the results in this paper. In all examples, we started with $p=I$ and $r_a=0$. We write numerical values with seven significant digits with two exceptions: for the theoretical value of the restoration entropy of the bouncing ball system and the Lorenz system and their estimates, we use more digits to highlight the difference, since the computed values are correct to 12 and 10 significant digits, respectively. If we use fewer than seven digits, then this is the exact value.%

\subsection{The H\'enon map}

The first example is the H\'enon system with standard parameters $a=1.4$ and $b=0.3$, which is given by%
\begin{align*}
  x(t+1) &= 1.4 - x(t)^2 + 0.3y(t), \\
	y(t+1) &= x(t).%
\end{align*}
It is known that the quadrilateral $K$ with the following corners is a trapping region \cite{Hen}:%
\begin{align*}
  A &= (-1.862,1.96),\quad B = (1.848,0.6267), \\
	C &= (1.743,-0.6533),\quad D = (-1.484,-2.3333),%
\end{align*}
In particular, $K$ is a compact forward-invariant set. We applied our algorithm to $\phi_{|K}$, where $\phi$ is the time-one map of the system and searched for a conformal metric with a polynomial of {maximal} degree 3.%

We have the following theoretical result on the restoration entropy for the H\'enon map from \cite[Thm.~16]{MPo}, which uses the two equilibria $e_{\pm} = (x_{\pm},x_{\pm})$, where%
\begin{equation}\label{eq_32}
  x_{\pm} = \frac{b-1 \pm \sqrt{(b-1)^2 + 4a}}{2}.%
\end{equation}

\begin{theorem}
For any compact forward-invariant set $K$ of $\phi$, we have the estimate%
\begin{equation}\label{eq_33}
  h_{\res}(K) \leq \log \Bigl(\sqrt{x_-^2 + b} - x_-\Bigr).%
\end{equation}
If $a > \frac{3}{4}(1 - b)^2$ and $e_+$ lies in the interior of $K$, then%
\begin{equation}\label{eq_34}
  h_{\res}(K) \geq \log \Bigl(\sqrt{x_+^2 + b} + x_+\Bigr).%
\end{equation}
If the intersection of the unstable manifold of $e_-$ with a sufficiently small neighborhood of $e_-$ lies in $K$, then \eqref{eq_34} holds with equality.%
\end{theorem}

Computing the upper estimate \eqref{eq_33} for the standard parameters yields%
\begin{equation*}
  h_{\res}(K) \leq 1.704793. 
\end{equation*}
The condition $a = 1.4 > \frac{3}{4}(1 - 0.3)^2$ is satisfied and moreover, it can easily be checked that $e_+$ lies in the interior of the trapping region $K$, hence%
\begin{equation*}
  h_{\res}(K) \geq  0.9439130. 
\end{equation*}
However, $e_-$ is outside of $K$, and thus it is not guaranteed that this estimate holds with equality {for the trapping region $K$}.%

In our computations, we set $N_1 = N_2 =$ 1,000 and $t_k = 16/k$. Our ansatz for the metric is $(x,y) \mapsto \rme^{r_a(x,y)}p$ with%
\begin{equation*}
  r_a(x,y) = a_1 x + a_2y  + a_{11}x^2 + a_{12}xy+ a_{22}y^2 + a_{112}x^2y + a_{122}xy^2 + a_{111}x^3 + a_{222}y^3.%
\end{equation*}

We started with $p=I$ and $r_a(x,y)=0$ and performed 4,000 iterations in 359\rm{s}, cf.~Figure \ref{fig1}. The best estimate of the restoration entropy was%
\begin{equation*}
  \hat{h}_{\res}(K) = 1.429359 
\end{equation*}
obtained in iteration 3,759 with%
\begin{equation*}
  p = \begin{pmatrix}
    1.631590  &-0.02982056  \\
-0.02982056 & 0.6134442
  \end{pmatrix}
\end{equation*}
and%
\begin{align*}
  r_a(x,y)&=0.1153877   x + 0.9776185 y+1.007788027   x^2 -0.05841912  xy
 +0.2099100   y^2\\ &\ \ \ -0.4110473  x^2y+ 0.05026169  xy^2 +0.01754927  x^3-0.007827144  y^3.
\end{align*}
\begin{figure}[ht]
   \centering
     \includegraphics[scale=0.23]{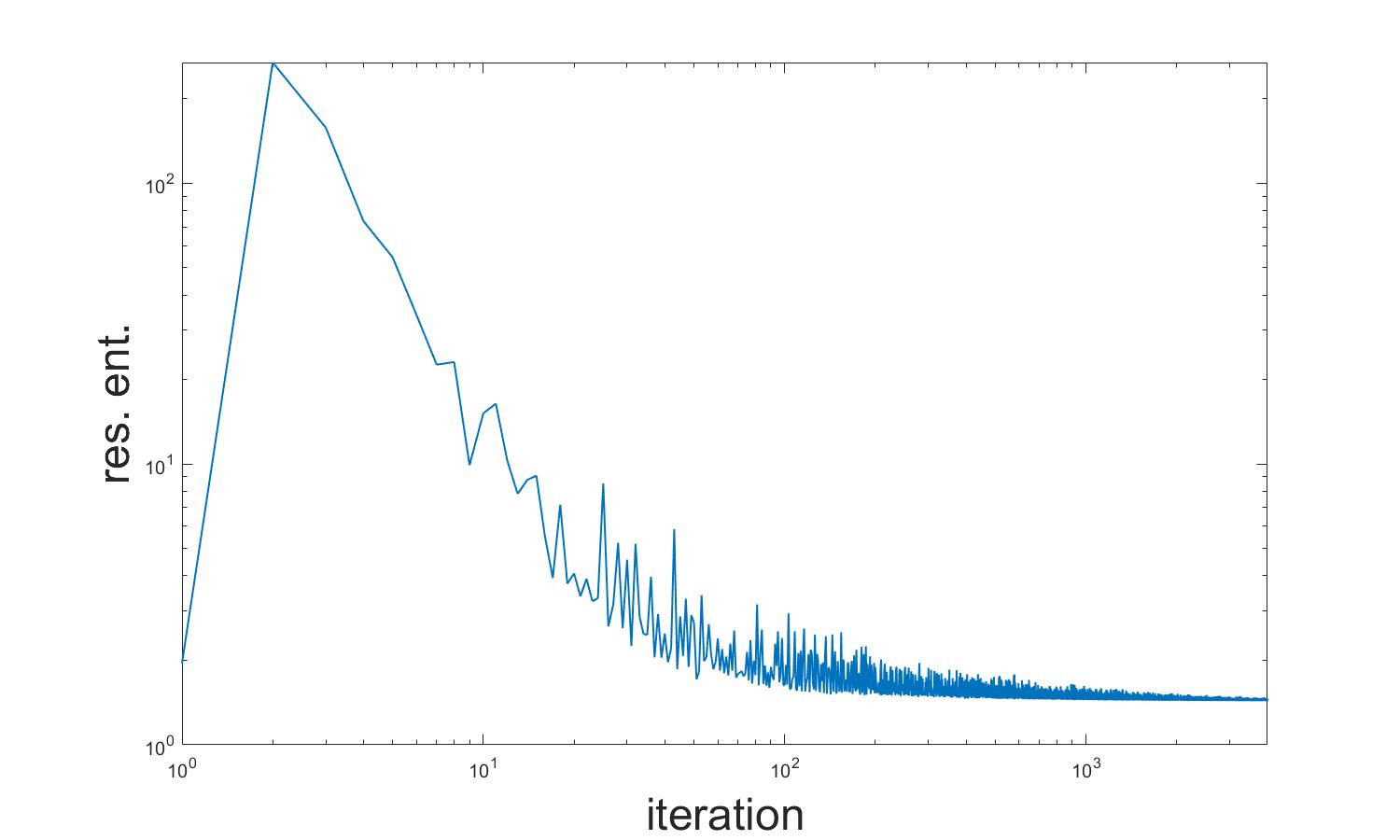}
    \caption{Restoration entropy for the H\'{e}non map as a function of iterations 1 to 4,001 on a loglog plot. The initial estimate for the restoration entropy with $p=I$ and $r_a(x,y)=0$ is 1.951141.} 
   \label{fig1}
  \end{figure}

{\bf Conclusion:} The best value obtained as an estimate for $h_{\res}(K)$ lies in the interval given by the theoretical upper and lower bounds. The estimate grows considerably in the first step and there are large variations in the sequence, that is not monotonically falling, although there is a clear trend downward.%

\subsection{Harmonically forced bouncing ball system}\label{BBsystsec}

The second example is a harmonically forced bouncing ball system, as discussed in \cite[Sec.~5]{MPo}. This system is given by the equations%
\begin{align*}
  x_1(t+1) &= x_1(t) + x_2(t), \\
	x_2(t+1) &= \gamma x_2(t) - \delta\cos(x_1(t) + x_2(t)),%
\end{align*}
where $\gamma \in (0,1)$ and $\delta>0$ are {positive} parameters. The physical meaning of these parameters is explained in more detail in \cite{MPo}. Due to the invariance under the transformation $x_1 \mapsto x_1 + 2\pi k$, $k\in\Z$, the state space of the system is typically taken to be the cylinder $\mathrm{S}^1 \tm \R$. The system has the trapping region $K := \mathrm{S}^1 \tm [-\delta(1-\gamma)^{-1},\delta(1-\gamma)^{-1}]$. For this set, \cite[Thm.~15]{MPo} yields%
\begin{equation*}
  h_{\res}(K) = \log\left(1 + \gamma + \delta + \sqrt{(1 + \gamma + \delta)^2 - 4 \gamma}\right) - 1.%
\end{equation*}

In our numerical case study, we put $\gamma=0.1$ and $\delta=2$, for which%
\begin{equation*}
  h_{\res}(K) = 1.617015883755.%
\end{equation*}
In our computations, we set $N_1=N_2=$ 1,000 and $t_k=1/k$. Since it has been shown in \cite[Proof of Thm.~15]{MPo} that a constant metric suffices for this system, our ansatz for the metric is $p$, i.e.~$(x,y)\mapsto \rme^{r_a(x,y)}p$ with $r_a(x,y) = 0$.%

We started with $p=I$ and performed 40 iterations in 2.2\rm{s}. The best estimate of the restoration entropy was%
\begin{equation*}
  \hat{h}_{\res}(K) = 1.617015883762 
\end{equation*}
obtained in iteration 31 with%
\begin{equation*}
p=\begin{pmatrix}
    1.362257 & 0.1134348 \\
    0.1134348 & 0.7435217  \\
  \end{pmatrix}
\end{equation*}.

\begin{figure}[ht]
   \centering
     \includegraphics[scale=0.23]{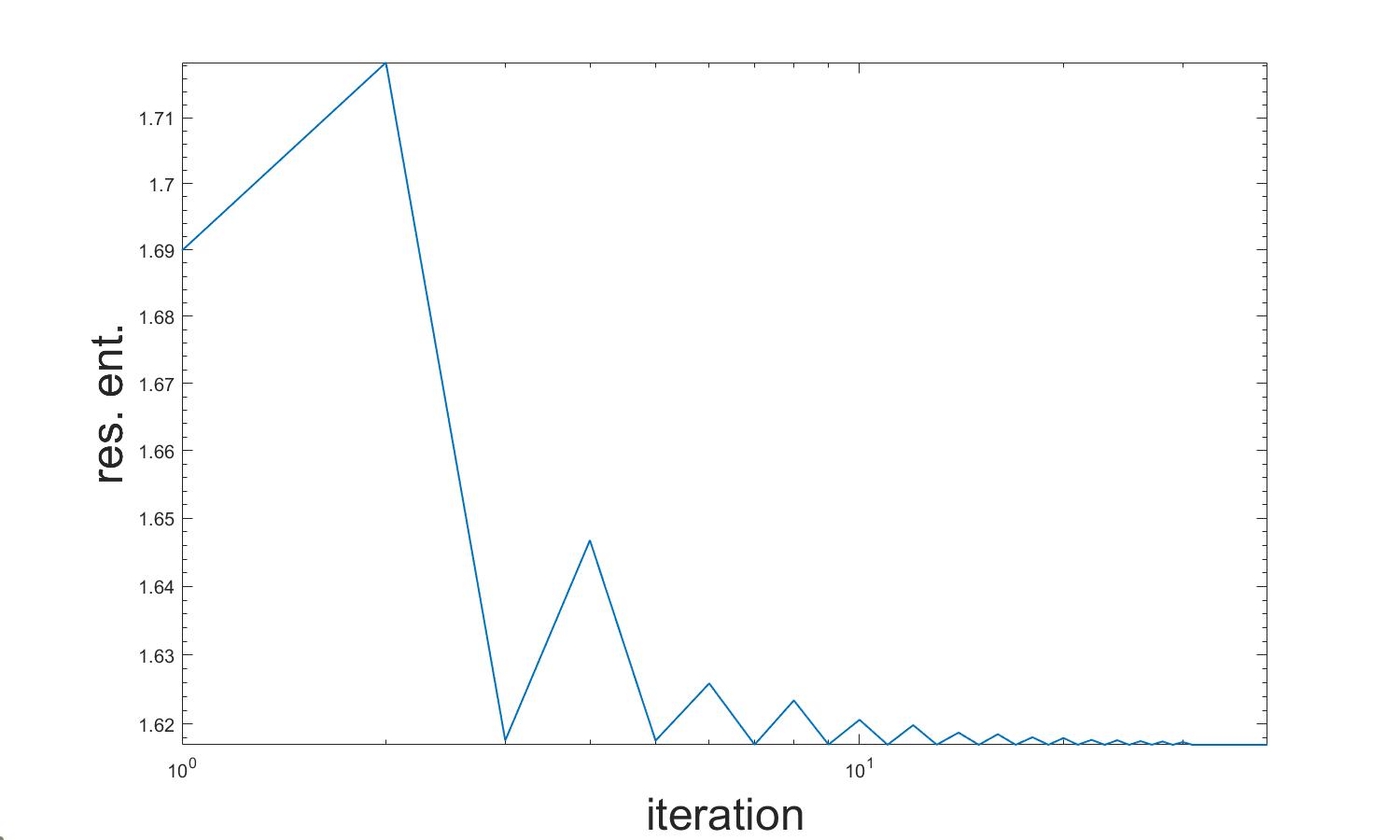}
    \caption{Restoration entropy for bouncing ball as a function of iterations 1 to 41 on a loglog plot. The initial estimate of the restoration entropy with $p=I$ and $r_a(x,y)=0$ is $1.689883$.  
     }
   \label{fig4}
\end{figure}

{\bf Conclusion:} The best value obtained as an estimate for $h_{\res}(K)$ is correct to 12 significant digits and the sequence obtained by the algorithm converges quickly.%

\subsection{The Lorenz system}

The third example is the three-dimensional continuous-time Lorenz system. This system is given by the equations %
\begin{align*}
  \dot x_1 &= \sigma(x_2 - x_1) \\
  \dot x_2 &= x_1(\rho-x_3)-x_2 \\
  \dot x_3 &= x_1x_2-\beta x_3
\end{align*}
where $\sigma$, $\rho$ and $\beta$ are parameters. We chose the standard values%
\begin{equation*}
  \sigma = 10,\ \ \ \rho=28\ \ \ \text{and}\ \ \ \beta=\frac{8}{3}.%
\end{equation*}
From \cite[Sec.~II.2.2]{BLR}, we know that the closed {ball} $K$ centered at $(0,0,\sigma+\rho)$ with radius $\sqrt{\beta/2}\,(\sigma+\rho)$ is forward-invariant for the system. Further, we know from \cite[Thm.~15]{MP2} and \cite[Thm.~4.3]{PMa} that%
\begin{equation*}
  h_{\res}(K) = \frac{1}{2\ln(2)} \left(\sqrt{(\sigma-1)^2+4\rho\sigma}-(\sigma+1)\right) = 17.063797967999616.
\end{equation*}
{Moreover, it can be shown by a direct computation that the metric}
\begin{align}\label{eq_lorenz_optimal_metric}
  P(x,y,z) &= \rme^{r(x,y,z)} \left( \begin{array}{ccc} \frac{\rho\sigma + (b-1)(\sigma - 1)}{\sigma^2} & -\frac{b-1}{\sigma} & 0 \\[0.2cm] -\frac{b-1}{\sigma} & 1 & 0 \\[0.2cm] 0 & 0 & 1 \end{array}\right)
\end{align}
where $r(x,y,z)$ is the quadratic polynomial%
\begin{align*}
  r(x,y,z) &= a \theta \left( \gamma_1 x^2 + \gamma_2 \left( y^2 + z^2 + \frac{(b-1)^2}{\sigma^2}x^2\right) + \gamma_3 z \right)
\end{align*}
with the constants%
\begin{align*}
  a &= \frac{\sigma}{\sqrt{r\sigma+(b-1)(\sigma-b)}} = 0.5849832,\\ 
  \theta &= \frac{1}{2\sqrt{(\sigma+1-2b)^2+(2\sigma/a)^2}} = 0.01442775,\\ 
	\gamma_3 &= -4\frac{\sigma}{ab} = 25.64176 ,\\ 
  \gamma_2 &= \frac{a}{2} = 0.2924916,\\ 
  \gamma_1 &= -\frac{2(\gamma_2/\sigma)(r\sigma - (b-1)^2) + \gamma_3 + (2/\sigma)a(b-1)}{2\sigma} = 0.4614867, 
\end{align*}
realizes this value. {The function $r$ was proposed in \cite[Sec.~IV.9.3]{BLR} as a Lyapunov-type function used in the estimation of the Lyapunov dimension of invariant sets for the Lorenz system.}

In the search for the maximum in (T1), we used spherical coordinates {for the ball with center $(0,0,\sigma+\rho)$} and $N_1 = 500$ for the radial distance $0$ to $\sqrt{\beta/2}\,(\sigma+\rho)$, $N_2=50$ for the azimuthal angle $0$ to $\pi$  and $N_3 = 100$ for the polar angle $0$ to $2\pi$. We set $t_k=2/k$ and our ansatz for the metric is $(x,y)\mapsto \rme^{r_a(x,y,z)}p$ with%
\begin{equation*}
  r_a(x,y,z)=a_1 x + a_2y+a_3z+a_{12}xy+a_{23}yz+a_{13}xz+a_{11}x^2+a_{22}y^2+a_{33}z^2.
\end{equation*}
We started with $p=I$ and $r_a(x,y,z)=0$ and performed 4,000 iterations in 912\rm{s}. The best estimate of the restoration entropy was%
\begin{equation*}
  \hat{h}_{\res}(K) = 17.06379797224715 
\end{equation*}
i.e., the theoretical value with ten correct significant digits, obtained in iteration 3,538 with%
\begin{equation*}
p=\begin{pmatrix}
1.633503   &0.1158726   &-0.00001799334  \\
0.1158726 &0.6206277  &0.00001341403   \\
-0.00001799334 &0.00001341403  &0.9996291
  \end{pmatrix}
\end{equation*}
and%
\begin{align*}
  r_a(x,y,z)&=10^{-5}\big(40.45413  \, x - 35.97656  \,y-21640.79  \,z-216.4097  \,xy\\
  &\ \ -21.43891  \, yz+1.316793  \,xz+405.8675  \,x^2+327.2187  \,y^2+304.7850 \,z^2\big).
\end{align*}
 \begin{figure}[ht]
   \centering
   \includegraphics[scale=0.23]{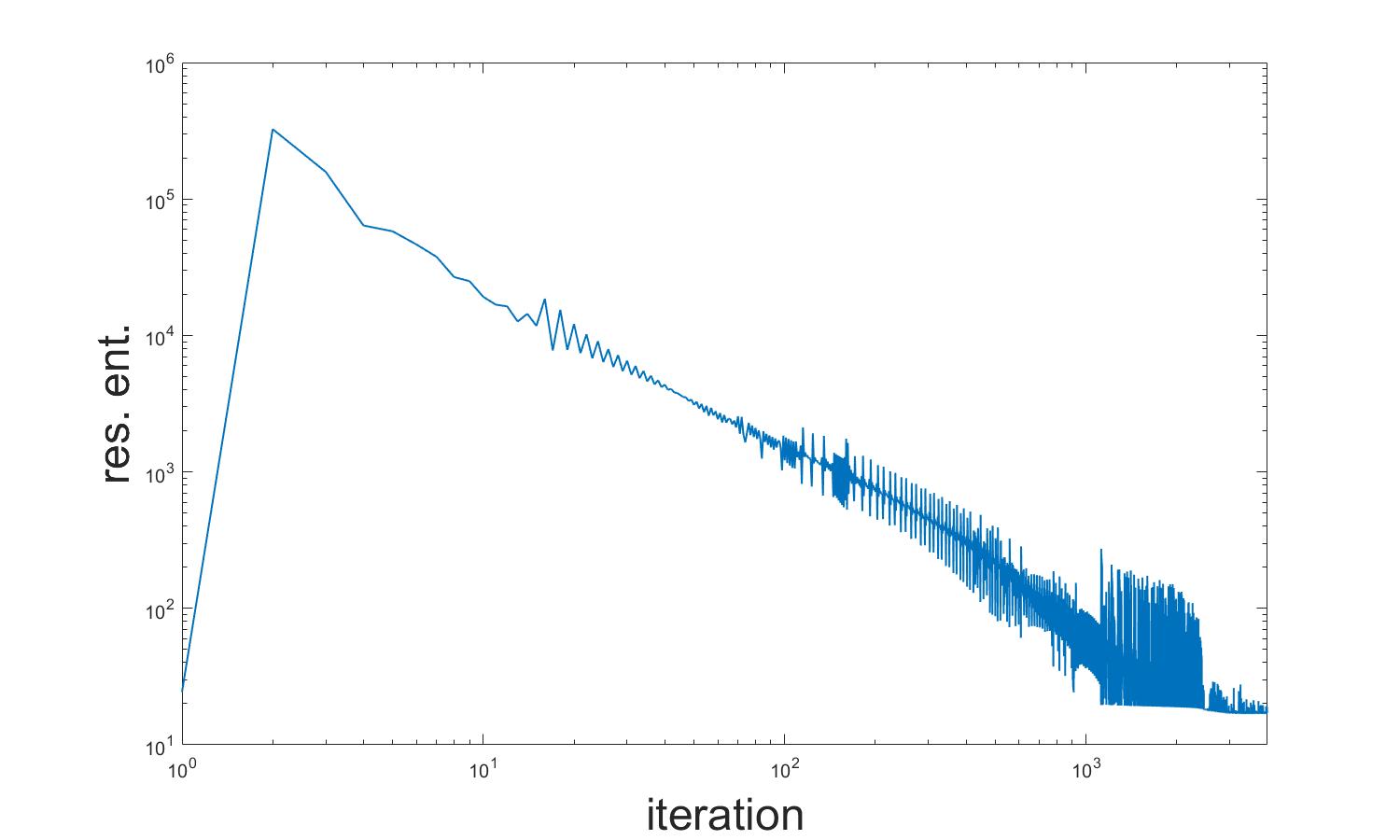}
    \caption{Restoration entropy for the Lorenz system as a function of iterations 1  to 4,001 on a loglog plot. The initial estimate for the restoration entropy with $p=I$ and $r_a(x,y,z)=0$ is $24.37586$. } 
   \label{figLorenz1}
  \end{figure}

{\bf Conclusion:} We obtain the theoretical value with ten significant digits in less than 4,000 iterations. However, the obtained metric is quite different from the one in \eqref{eq_lorenz_optimal_metric}, which with our values for the parameters is%
\begin{equation*}
  P(x,y,z) =\rme^{r(x,y,z)}\left(\begin{array}{ccc} 2.95 & -0.1666667& 0 \\ -0.1666667& 1 & 0 \\ 0 & 0 & 1 \end{array}\right)%
\end{equation*}
with%
\begin{align*}
  r(x,y,z) \approx 10^{-2}\big(-21.64162\, z + 0.3963516\, x^2 +0.2468626\, y^2+0.2468626\, z^2\big).%
\end{align*}
However, in both metrics the term $z$ in the polynomial dominates.

\section{Conclusions and future work}\label{sec_future}

In this paper, we have introduced a numerical algorithm to compute upper bounds for the restoration entropy of a dynamical system. This quantity characterizes the smallest channel capacity (or data rate) above which the system can be regularly or finely observed over a digital channel by a remote observer. Our algorithm also delivers a conformal Riemannian metric which can be used for the explicit design of an observer.%

There are a number of open questions about the proposed subgradient algorithm, leading to topics for future work:%
\begin{itemize}
\item It is an open question in which cases conformal metrics are sufficient to approximate the restoration entropy. The fact that the algorithm works so well for the tested examples might be related to the low dimensions of their state spaces or particular dynamical properties. Hence, one should try to extend the algorithm to larger classes of Riemannian metrics, which can still be described by finitely many parameters.%
\item The ordinary subgradient algorithm on Euclidean space can be improved in several ways. For instance, \emph{bundle methods} \cite{Lea} allow to obtain better convergence properties by using more local information about the function to be minimized in order to obtain directions of descent (instead of arbitrary subgradients {which, in general, are not directions in which the function decreases, as can be seen clearly from our plots in Section \ref{sec_examples}}). Possibly, such methods can be extended to the subgradient algorithm on Riemannian manifolds.%
\item If one is only interested in the computation of restoration entropy (and not on the computation of an associated Riemannian metric on $K$), one can reduce the complexity of the maximization task involved in our algorithms by first computing an approximation of the recurrent set within $K$. Indeed, this follows from \cite[Cor.~A.8]{Mor}, which shows that the maximal growth rate of a subadditive cocycle is attained at a recurrent point. The recurrent set is contained in the chain-recurrent set and for the outer approximation of the latter, there exist well-established algorithms \cite{DJ}.%
\item We can reduce the dimension of the parameter space $P(d,n)$ by observing that%
\begin{equation*}
  \Sigma(sP) = \Sigma(P) \mbox{\quad for all\ } s > 0,\ P \in C^0(K,\SC^+_n).%
\end{equation*}
Hence, we can always set the constant term in the polynomial $r_a(x)$ equal to zero and, moreover, we can require that $\det p = 1$ (or any other constant value). This reduces the dimension of the parameter space by $2$. Observe that the reduced parameter space is still geodesically convex, since convex combinations of polynomials with vanishing constant term also have vanishing constant term and $\det p = \det q = 1$ implies {$\det(p \#_t\, q) = 1$ for all $t\in[0,1]$}. A further reduction of the parameter space dimension might be possible by ``ignoring'' or ``factoring out'' the matrices $p \in \SC^+_n$ with%
\begin{equation*}
  p^{\frac{1}{2}}A(x)p^{-\frac{1}{2}} = A(x) \mbox{\quad for all\ } x \in K.%
\end{equation*}
However, it is not so clear how to do this formally.%
\item One of the drawbacks of the proposed algorithm is that it involves a nonlinear maximization problem, which usually does not have good properties such as convexity. It is thus an important question for future investigations how a lack of accuracy in the solution of this maximization problem affects the result of the subgradient algorithm. In our examples, the estimate was not overly sensitive to the density of the grid, where we searched for the maximum, given that it was reasonably high.%
\item Our algorithm seems to have a great potential to be applicable to other problems, including the computation of contraction metrics for exponentially stable equilibria and periodic orbits \cite{Gie15,Gie19}, the estimation of the dimension of invariant sets {\cite{BLR,PN}}, and the approximation of extremal Lyapunov exponents \cite{Sko}. We leave the {study} of such extensions to future investigations.%
\end{itemize}


\appendix

\section{A lemma on sectional curvature}

{
To guarantee that the subgradient algorithm converges, we need to make sure that the product manifold $\R^N \tm \SC^+_n$ has sectional curvature bounded from below. Obviously, the Euclidean factor $\R^N$ has vanishing sectional curvature. It is further well-known that the sectional curvature of $\SC^+_n$ is uniformly bounded from below. The following lemma thus guarantees the desired curvature bound for $\R^N \tm \SC^+_n$. We assume that the reader is familiar with basic concepts and notation used in Riemannian geometry (two standard references are \cite{DCa,GHL}).}%

{
\begin{lemma}\label{lem_product_curv}
Let $M_1,M_2$ be two Riemannian manifolds of non-positive sectional curvature and let $M := M_1 \tm M_2$ be equipped with the product metric. If $k_1,k_2 \leq 0$ are lower bounds on the sectional curvature of $M_1$ and $M_2$, respectively, then $k_1 + k_2$ is a lower bound on the sectional curvature of $M$.
\end{lemma}}

{
\begin{proof}
It is well-known that the Riemannian curvature tensor of $M$ satisfies the identity%
\begin{equation}\label{eq_ct_id}
  R(X_1+X_2,Y_1+Y_2,Z_1+Z_2,W_1+W_2) = R_1(X_1,Y_1,Z_1,W_1) + R_2(X_2,Y_2,Z_2,W_2),%
\end{equation}
for any $X_i,Y_i,Z_i$ and $W_i$, where $R_i$ is the curvature tensor of $M_i$, $i = 1,2$.
Now, fix $p = (p_1,p_2) \in M$ and a $2$-dimensional subspace $\Pi \subset T_pM = T_{p_1}M_1 \tm T_{p_2}M_2$. Let $(X,Y)$ be an orthonormal basis of $\Pi$ and split $X = (X_1,X_2)$, $(Y_1,Y_2)$ with $X_1,Y_1 \in T_{p_1}M_1$ and $X_2,Y_2 \in T_{p_2}M_2$. From \eqref{eq_ct_id}, it follows that the sectional curvature of $\Pi$ satisfies%
\begin{equation*}
  K(\Pi) = R(X,Y,Y,X) = R_1(X_1,Y_1,Y_1,X_1) + R_2(X_2,Y_2,Y_2,X_2).%
\end{equation*}
Now we distinguish four cases:%
\begin{enumerate}
\item[(i)] $X_1,Y_1$ are linearly independent and $X_2,Y_2$ are linearly independent. In this case, let $\Pi_i$ be the span of $X_i,Y_i$ for $i=1,2$. Then%
\begin{equation}\label{eq_c1}
  K(\Pi) = (|X_1|^2 |Y_1|^2 - \langle X_1,Y_1 \rangle) K(\Pi_1) + (|X_2|^2 |Y_2|^2 - \langle X_2,Y_2 \rangle) K(\Pi_2).%
\end{equation}
\item[(ii)] $X_1,Y_1$ are linearly independent, but $X_2,Y_2$ are not. Then%
\begin{equation}\label{eq_c2}
  K(\Pi) = (|X_1|^2 |Y_1|^2 - \langle X_1,Y_1 \rangle) K(\Pi_1).%
\end{equation}
\item[(iii)] $X_2,Y_2$ are linearly independent, but $X_1,Y_1$ are not. Then%
\begin{equation}\label{eq_c3}
  K(\Pi) = (|X_2|^2 |Y_2|^2 - \langle X_2,Y_2 \rangle) K(\Pi_2).%
\end{equation}
\item[(iv)] Neither $X_1,Y_1$ nor $X_2,Y_2$ are linearly independent. Then%
\begin{equation}\label{eq_c4}
  K(\Pi) = 0.%
\end{equation}
\end{enumerate}
By assumption we have $k_1 \leq K(\Pi_1) \leq 0$ and $k_2 \leq K(\Pi_2) \leq 0$. Since $|X| = |Y| = 1$, we further have $|X_i|,|Y_i| \leq 1$ for $i = 1,2$. Then%
\begin{equation*}
  0 \leq |X_i|^2 |Y_i|^2 - \langle X_i,Y_i \rangle = |X_i|^2 |Y_i|^2 (1 - \cos^2 \angle (X_i,Y_i)) \leq 1.%
\end{equation*}
Thus, \eqref{eq_c1} implies $K(\Pi) \geq k_1 + k_2$, and so do \eqref{eq_c2}, \eqref{eq_c4} and \eqref{eq_c4}.
\end{proof}
}

\section*{Acknowledgements}

The first author thanks Alexander Pogromsky and Alexey Matveev for enlightening discussions about restoration entropy and its computation, as well as Jost Eschenburg and Peter Quast for answering his questions about the geometry of $\SC^+_n$.%

\end{document}